\documentclass[12pt]{article}

\usepackage[english]{babel}
\usepackage[utf8]{inputenc}
\usepackage[T1]{fontenc}
\usepackage{xcolor}
\usepackage{mathtools}
\usepackage{lipsum}
\usepackage{multicol}
\makeatletter
\newcounter {subsubsubsection}[subsubsection]
\renewcommand\thesubsubsubsection{\thesubsubsection .\@arabic\c@subsubsubsection}
\newcommand\subsubsubsection{\@startsection{subsubsubsection}{4}{\z@}%
                                     {-3.25ex\@plus -1ex \@minus -.2ex}%
                                     {1.5ex \@plus .2ex}%
                                     {\normalfont\normalsize\bfseries}}
\renewcommand\paragraph{\@startsection{paragraph}{5}{\z@}%
                                    {3.25ex \@plus1ex \@minus.2ex}%
                                    {-1em}%
                                    {\normalfont\normalsize\bfseries}}
\renewcommand\subparagraph{\@startsection{subparagraph}{6}{\parindent}%
                                       {3.25ex \@plus1ex \@minus .2ex}%
                                       {-1em}%
                                      {\normalfont\normalsize\bfseries}}
\newcommand*\l@subsubsubsection{\@dottedtocline{4}{10.0em}{4.1em}}
\renewcommand*\l@paragraph{\@dottedtocline{5}{10em}{5em}}
\renewcommand*\l@subparagraph{\@dottedtocline{6}{12em}{6em}}
\newcommand*{\subsubsubsectionmark}[1]{}
\makeatother

\usepackage{hyperref}
\newcommand\restr[2]{{
  \left.\kern-\nulldelimiterspace 
  #1 
  \vphantom{\big|} 
  \right|_{#2} 
  }}
\makeatletter
\def\toclevel@subsubsubsection{4}
\def\toclevel@paragraph{5}
\def\toclevel@subparagraph{6}
\makeatother

\usepackage[top=1.3cm, bottom=2.0cm, outer=2.5cm, inner=2.5cm, heightrounded,
marginparwidth=1.5cm, marginparsep=0.4cm, margin=2.5cm]{geometry}

\usepackage{graphicx} 
\usepackage{amsmath}  
\usepackage{amsfonts} %
\usepackage{amssymb}  %
\usepackage{amsthm}
\usepackage{bm}
\usepackage{hyperref}
\hypersetup{
    colorlinks = false,
    linkbordercolor = {red},
}

\usepackage[
backend=biber,
style=alphabetic,
sorting=anyvt,
natbib=true
]{biblatex}
\usepackage{amsmath,amsthm,amssymb}
\usepackage{enumitem}

\newcommand{\C}{\mathbb{C}}

\DeclareFontFamily{OT1}{pzc}{}
\DeclareFontShape{OT1}{pzc}{m}{it}{<-> s * [1.10] pzcmi7t}{}
\DeclareMathAlphabet{\mathpzc}{OT1}{pzc}{m}{it}

\newcommand{\overbar}[1]{\mkern 1.5mu\overline{\mkern-1.5mu#1\mkern-1.5mu}\mkern 1.5mu}
\newcommand{\norm}[1]{\left\lVert#1\right\rVert}
\newcommand{\abs}[1]{\left\lvert#1\right\rvert}

\newtheorem{theorem}{Theorem}
[section]
\newtheorem{thmx}{Theorem}

\theoremstyle{definition}
\newtheorem{definition}{Definition}[section]

\theoremstyle{plain}

\newtheorem{proposition}[theorem]{Proposition}

\theoremstyle{remark}

\theoremstyle{definition}

\newcommand{\del}{\partial}
\newcommand{\dbar}{\bar{\partial}}
\newcommand{\ddbar}{\del\dbar}

\newcommand{\ricci}{\mathrm{\mathbf{Ric}}}

\newcommand{\holosections}{\Gamma_{\mathcal{O}}}

\newcommand{\HH}{\mathcal{H}^2}
\usepackage{tikz}
\usetikzlibrary{shapes.geometric, arrows}
\tikzstyle{startstop} = [rectangle, rounded corners, minimum width=3cm, minimum height=1cm,text centered, draw=black]
\tikzstyle{process} = [rectangle, minimum width=3cm, minimum height=1cm, text centered, draw=black]

\tikzstyle{arrow} = [thick,->,>=stealth]

\makeatletter
\long\def\pgfshapeaddanchor#1#2{%
{%
  \def\pgf@sm@shape@name{#1}%
  \let\anchor=\pgf@sh@anchor%
  #2}%
}
\pgfshapeaddanchor{rectangle}{%
  \anchor{west north west}{%
    \pgf@process{\northeast}%
    \pgf@ya=1.5\pgf@y%
    \pgf@process{\southwest}%
    \pgf@y.5\pgf@y%
    \advance\pgf@y by \pgf@ya%
    \pgf@y=.5\pgf@y%
  }%
  \anchor{north north west}{%
    \pgf@process{\southwest}%
    \pgf@xa=1.5\pgf@x%
    \pgf@process{\northeast}%
    \pgf@x.5\pgf@x%
    \advance\pgf@x by \pgf@xa%
    \pgf@x=.5\pgf@x%
  }%
  \anchor{west south  west}{%
    \pgf@process{\northeast}%
    \pgf@ya=.5\pgf@y%
    \pgf@process{\southwest}%
    \pgf@y=1.5\pgf@y%
    \advance\pgf@y by \pgf@ya%
    \pgf@y=.5\pgf@y%
  }%
  \anchor{south south  west}{%
    \pgf@process{\northeast}%
    \pgf@xa=.5\pgf@x%
    \pgf@process{\southwest}%
    \pgf@x=1.5\pgf@x%
    \advance\pgf@x by \pgf@xa%
    \pgf@x=.5\pgf@x%
  }%
}

\makeatother

\title{Twisted Nakano-Positivity of Fields of Hilbert Spaces}
\author{E. M. Ainasse}
\date{\today}
\AtEndDocument{\bigskip{\footnotesize%
   \begin{multicols}{2}
  \textsc{Mathematics Department\\
  Stony Brook University\\
  Stony Brook, NY, 11794-3651, USA}\\
 \vfill\null
 \columnbreak
 \textsc{Department of Applied Mathematics \& Statistics\\
 Stony Brook University\\
 Stony Brook, NY 11794-3600, USA}
  \end{multicols}
  \par  
  \textit{E-mail address}: \texttt{elmehdi.ainasse@stonybrook.edu}}}
  
\begin{document}
\maketitle

\begin{abstract}
In \citep{ainasse2021twisted}, we generalized Berndtsson's Nakano-positivity by retaining the same consequences under weaker hypotheses. In this article, we propose to further generalize our ``twisted'' Nakano-positivity theorem to fields of Hilbert spaces associated to possibly unbounded Stein manifolds. We achieve this result using exhaustion arguments. As direct applications, we prove $\log$-plurisubharmonic variation results for a certain class of non-trivial families of Stein manifolds.
\end{abstract}

\section{Introduction and description of the main results}\label{introduction}

Let $X$ be a Stein manifold equipped with a Kähler metric $g$, let $V \rightarrow X$ be a holomorphic vector bundle, and let $U$ be a domain in $\C^m$ containing the origin. Let $h$ be a metric for the pullback bundle $\pi^{*}_{\overbar{X}} V \rightarrow U \times \overbar{X}$, where $\pi_{\overbar{X}} : U \times \overbar{X} \rightarrow \overbar{X}$ is the canonical projection, and define $h^{[t]} := i^{*}_t h$ where $i_t$ denotes the inclusion map $\overbar{X} \hookrightarrow \{t\} \times \overbar{X}$. This defines a smooth family $\left\{h^{[t]}\right\}_{t \in U}$ of smooth Hermitian metrics for $V \rightarrow X$, and we can then consider the space $L^2\left(X,h^{[t]}\right)$ of sections $f$ of $V \rightarrow X$ whose norm $\abs{f}_{h^{[t]}}$, with respect to the metric $h^{[t]}$, is square-integrable on $X$ with respect to the volume form $dV_g$ induced by the Kähler metric $g$. For each $t \in U$, we can further consider the space $\HH\left(X,h^{[t]}\right)$ of holomorphic sections of $V \rightarrow X$ in $L^2\left(X,h^{[t]}\right)$.

When $X$ is a relatively compact submanifold of an ambient Stein Kähler manifold $(Y,g)$, $\HH\left(X,h^{[t]}\right)$ is a closed subspace of $L^2\left(X,h^{[t]}\right)$ -- hence a Hilbert space -- and the Hilbert spaces in the collection $\left\{\HH\left(X,h^{[t]}\right)\right\}_{t \in U}$ have equivalent norms. Thus, the underlying vector spaces of the Bergman spaces $\mathcal{H}^2\left(X,h^{[t]}\right)$ are equal as subspaces of the space $\holosections(X,V)$. By fixing $\mathcal{H}^2_0 := \HH\left(X,h^{[0]}\right)$, we can define the bundle $E_h$ of infinite rank over $U$ with total space $U \times \mathcal{H}^2_0$, whose fiber over $t \in U$ is $\{t\} \times \mathcal{H}^2_0 \cong \HH_t =: \mathcal{H}^2\left(X,h^{[t]}\right)$. It is a trivial Hilbert bundle equipped with the non-trivial Hermitian metric $\left(\cdot,\cdot\right)_{h^{[t]}}$, varying in $t$, induced by the $L^2$-norm on $\mathcal{H}^2_t$.\\

In \citep{ainasse2021twisted} we proved a twisted Nakano-positivity theorem for such Hilbert bundles.

\begin{theorem}\label{thm-nakano-positivity-smooth-bounded-stein}\emph{(\citep[Theorem A]{ainasse2021twisted})}
Let $X$ be an $n$-dimensional relatively compact pseudoconvex subdomain of an ambient Stein Kähler manifold $(Y,g)$. Let $V \rightarrow \overbar{X}$ be a holomorphic vector bundle. Let $U \subset \C^m$ be a domain, and let $\left\{h^{[t]}\right\}_{t \in U}$ be a family of smooth Hermitian metrics for $V \rightarrow \overbar{X}$. Let $\delta > 0$ and let $\eta$ be a smooth function on $Y$. If $\Xi_{\delta,\eta}(h) >_{\mathrm{Griff}} 0$ and
$$\dbar_X\left(\left(h^{[t]}\right)^{-1}\del_X h^{[t]}\right) + \left(\ricci(g) + 2\del_X\dbar_X\eta-(1+\delta)\del_X\eta\wedge\dbar_X\eta\right) \otimes \mathrm{Id}_V >_{\mathrm{Nak}} 0,$$
for each $t \in U$, then the holomorphic Hermitian bundle $\left(E_h,\left(\cdot,\cdot\right)_{h^{[t]}}\right)$ is Nakano positive. Moreover, if either $\Xi_{\delta,\eta}(h) \geq_{\mathrm{Griff}} 0$ or $$\dbar_X\left(\left(h^{[t]}\right)^{-1}\del_X h^{[t]}\right) + \left(\ricci(g) + 2\del_X\dbar_X\eta-(1+\delta)\del_X\eta\wedge\dbar_X\eta\right) \otimes \mathrm{Id}_V \geq_{\mathrm{Nak}} 0,$$

for each $t \in U$, then $\left(E_h,\left(\cdot,\cdot\right)_{h^{[t]}}\right)$ is Nakano semipositive.
\end{theorem}

We recall that given a smooth Hermitian metric $h$ for $V \rightarrow \overbar{X}$, $\Theta_{\delta}(h)$ is locally defined as
\begin{align*}
&\sum_{1 \leq j,k \leq m}\dfrac{\del}{\del \bar{t}_k}\left(h^{-1}\dfrac{\del h}{\del t_j}\right) d\bar{t}_k \wedge dt_j + \sum_{\substack{1 \leq j \leq m \\ 1 \leq \mu \leq n}}\dfrac{\del}{\del\bar{z}_\mu}\left(h^{-1}\dfrac{\del h}{\del t_j}\right) d\bar{z}_\mu \wedge dt_j\\
&\, \, \, \, \, \, + \sum_{\substack{1 \leq k \leq m \\ 1 \leq \nu \leq n}}\dfrac{\del}{\del\bar{t}_k}\left(h^{-1}\dfrac{\del h}{\del z_{\nu}}\right) d\bar{t}_k \wedge dz_\nu + \dfrac{\delta}{1+\delta}\sum_{1 \leq \nu, \mu \leq n}\dfrac{\del}{\del\bar{z}_\mu}\left(h^{-1}\dfrac{\del h}{\del z_{\nu}}\right)d\bar{z}_{\mu} \wedge dz_{\nu},
\end{align*}
where $\delta > 0$. We can also express $\Theta_{\delta}(h)$ as
$$\Theta_{\delta}(h) = \Theta(h)-\dfrac{1}{1+\delta}\pi^{*}_X\Theta\left(h^{[t]}\right).$$
As a block matrix split with respect to the product structure $U \times X$ for our holomorphic trivial family, $\Theta_{\delta}(h)$ has the form
$$\Theta_{\delta}(h) = \begin{pmatrix} \dbar_U\left(h^{-1}\del_U h\right) & \dbar_X\left(h^{-1}\del_U h\right) \\ \dbar_U\left(h^{-1}\del_X h\right) & \dfrac{\delta}{1+\delta}\dbar_X\left(h^{-1}\del_X h\right)\end{pmatrix}.$$

Let $\eta$ be a smooth function on $Y$ and define the \textit{twisted} curvature operator $\Xi_{\delta,\eta}$ by
$$\Xi_{\delta,\eta}(h) := \Theta_{\delta}(h) + \dfrac{\delta}{1+\delta}\pi^{*}_X\left(\left(\ricci(g) + 2\ddbar_X\eta-(1+\delta)\del_X\eta\wedge\dbar_X\eta\right)\otimes \mathrm{Id}_V\right).$$

The operator $\Xi_{\delta,\eta}$ can also be represented as the block matrix
$$\Xi_{\delta,\eta}(h) = \begin{pmatrix}
\dbar_U\left(h^{-1}\del_U h\right) & \dbar_X\left(h^{-1}\del_U h\right)\\
\dbar_U\left(h^{-1}\del_X h\right) & \dfrac{\delta}{1+\delta}\left(\dbar_X\left(h^{-1}\del_X h\right) + \ricci(g) + 2\del_X\dbar_X\eta-(1+\delta)\del_X\eta\wedge\dbar_X\eta\right)
\end{pmatrix}.$$
Noting that
$$2\del_X\dbar_X\eta-(1+\delta)\del_X\eta\wedge\dbar_X\eta = \dfrac{4e^{\frac{1+\delta}{2}\eta}}{1+\delta}\del_X\dbar_X\left(-e^{-\frac{1+\delta}{2}\eta}\right),$$ we may also rewrite $\Xi_{\delta,\eta}(h)$ as
$$\Xi_{\delta,\eta}(h) = \begin{pmatrix}
\dbar_U\left(h^{-1}\del_U h\right) & \dbar_X\left(h^{-1}\del_U h\right)\\
\dbar_U\left(h^{-1}\del_X h\right) & \dfrac{\delta}{1+\delta}\left(\dbar_X\left(h^{-1}\del_X h\right) + \ricci(g) + \dfrac{4e^{\frac{1+\delta}{2}\eta}}{1+\delta}\del_X\dbar_X\left(-e^{-\frac{1+\delta}{2}\eta}\right)\right)
\end{pmatrix}.$$

Note that $\Xi_{\delta,\eta}(h) \geq_{\mathrm{Griff}} 0$ if and only if $\Xi_{\delta,\eta}(h) \geq 0$ as a block matrix.\\

In this article, we are interested in the case when $X$ is not necessarily a relatively compact submanifold of an ambient Stein manifold. Using our theorem \citep{ainasse2021twisted} in combination with exhaustion arguments, we will in fact prove more general curvature positivity results in the case where $X$ is a possibly unbounded Stein manifold. In the sequel, we will simply refer to such a manifold as an unbounded Stein manifold In this situation, $E_h$ may no longer have the structure of a vector bundle as the fibers may fail to be isometric, resulting in the absence of local triviality. In other words, $E_h$ is simply a family of Hilbert spaces indexed by $t$ or in other words, a \textit{field of Hilbert spaces} (see \citep[Definition 2.2.1]{Lempert-Szoke-2014}).
Using alternative characterizations of Griffiths (semi)positivity and Nakano (semi)positivity, we prove the following result.

\begin{thmx}\label{nakano-positivity-general-stein}
Let $(X,g)$ be any Stein Kähler manifold, let $U$ be a domain in $\C^m$, and let $V \rightarrow X$ be a holomorphic vector bundle. Let $\left\{h^{[t]}\right\}_{t \in U}$ be a family of smooth Hermitian metrics for $V \rightarrow X$, and let $\left(E_h,\left(\cdot,\cdot\right)_{h^{[t]}}\right)$ be the holomorphic Hermitian field of Hilbert spaces whose fiber at $t$ is $\mathcal{H}^2_t := \mathcal{H}^2\left(X,h^{[t]}\right)$. Let $\delta > 0$ and $\eta$ be a smooth function on $X$. If $\Xi_{\delta,\eta}(h) >_{\mathrm{Griff}} 0$ and 
$$\dbar_X\left(\left(h^{[t]}\right)^{-1}\del_X h^{[t]}\right) + \left(\ricci(g) + 2\del_X\dbar_X\eta-(1+\delta)\del_X\eta\wedge\dbar_X\eta\right) \otimes \mathrm{Id}_V >_{\mathrm{Nak}} 0,$$
for each $t \in U$, then the holomorphic Hermitian bundle $\left(E_h,\left(\cdot,\cdot\right)_{h^{[t]}}\right)$ is Nakano positive in the sense of definition \ref{nakano-positivity-general-definition}. Moreover, if either $\Xi_{\delta,\eta}(h) \geq_{\mathrm{Griff}} 0$ or
$$\dbar_X\left(\left(h^{[t]}\right)^{-1}\del_X h^{[t]}\right) + \left(\ricci(g) + 2\del_X\dbar_X\eta-(1+\delta)\del_X\eta\wedge\dbar_X\eta\right) \otimes \mathrm{Id}_V \geq_{\mathrm{Nak}} 0,$$
then $\left(E_h,\left(\cdot,\cdot\right)_{h^{[t]}}\right)$ is Nakano semipositive in the sense of definition \ref{nakano-positivity-general-definition}.
\end{thmx}

Prior to proving Theorem \ref{nakano-positivity-general-stein}, we also show that the holomorphic Hermitian field of Hilbert spaces $\left(E_h,\left(\cdot,\cdot\right)_{h^{[t]}}\right)$ is Griffiths positive in a general sense by using a different approach than the one used in our proof of Theorem \ref{nakano-positivity-general-stein}.

\begin{thmx}\label{griffiths-positivity-general-stein}
Let $(X,g)$ be any Stein Kähler manifold, let $U$ be a domain in $\C^m$, and let $V \rightarrow X$ be a holomorphic vector bundle. Let $\left\{h^{[t]}\right\}_{t \in U}$ be a family of smooth Hermitian metrics for $V \rightarrow X$ and let $\left(E_h,\left(\cdot,\cdot\right)_{h^{[t]}}\right)$ be the holomorphic Hermitian field of Hilbert spaces whose fiber at $t$ is $\mathcal{H}^2_t := \mathcal{H}^2\left(X,h^{[t]}\right)$. Let $\delta > 0$ and $\eta$ be a smooth function on $X$. If $\Xi_{\delta,\eta}(h) >_{\mathrm{Griff}} 0$ and $$\dbar_X\left(\left(h^{[t]}\right)^{-1}\del_X h^{[t]}\right) + \left(\ricci(g) + 2\del_X\dbar_X\eta-(1+\delta)\del_X\eta\wedge\dbar_X\eta\right) \otimes \mathrm{Id}_V >_{\mathrm{Nak}} 0,$$ 
for each $t \in U$, then the holomorphic Hermitian bundle $\left(E_h,\left(\cdot,\cdot\right)_{h^{[t]}}\right)$ is Griffiths positive in the sense of definition \ref{griffiths-positivity-definition-general}. Moreover, if
either $\Xi_{\delta,\eta}(h) \geq_{\mathrm{Griff}} 0$ or $$\dbar_X\left(\left(h^{[t]}\right)^{-1}\del_X h^{[t]}\right) + \left(\ricci(g) + 2\del_X\dbar_X\eta-(1+\delta)\del_X\eta\wedge\dbar_X\eta\right) \otimes \mathrm{Id}_V \geq_{\mathrm{Nak}} 0,$$
for each $t \in U$, then $\left(E_h,\left(\cdot,\cdot\right)_{h^{[t]}}\right)$ is Griffiths semipositive in the sense of definition \ref{griffiths-positivity-definition-general}.
\end{thmx}

As Berndtsson's work has been the main inspiration for our work, in the spirit of generalizing some of the results of Berndtsson's \textit{complex Brunn-Minkowski theory} (\citep{berndtsson2018complex}), we use our general curvature positivity theorems to establish $\log$-plurisubharmonic variation results similar to those of Berndtsson. Namely, we prove $\log$-plurisubharmonic variation results for Bergman kernels and families of compactly supported measures for trivial families of unbounded Stein manifolds. As in the case of Berndtsson's complex Brunn-Minkowski theory, these results really only require the Griffiths positivity of the holomorphic Hermitian field of Hilbert spaces $\left(E_h,\left(\cdot,\cdot\right)_{h^{[t]}}\right)$ and follow almost immediately.

\begin{thmx}\label{bergman-kernel-trivial-theorem}
Let $(X,g)$ be any Stein Kähler manifold, let $U \subset \C^m$ be a domain, and let $V \rightarrow X$ be a holomorphic vector bundle. Let $\left\{h^{[t]}\right\}_{t\in U}$ be a family of smooth Hermitian metrics for $V \rightarrow X$. Let $\delta > 0$ and $\eta$ be smooth function on $X$. Denote by $K_t$ the Bergman kernel for the projection $L^2\left(X,h^{[t]}\right) \rightarrow \HH\left(X,h^{[t]}\right)$. 

If $\Xi_{\delta,\eta}(h) >_{\mathrm{Griff}} 0$ and $$\dbar_X\left(\left(h^{[t]}\right)^{-1}\del_X h^{[t]}\right) + \left(\ricci(g) + 2\del_X\dbar_X\eta-(1+\delta)\del_X\eta\wedge\dbar_X\eta\right) \otimes \mathrm{Id}_V >_{\mathrm{Nak}} 0,$$ for each $t \in U$, 
then the family of possibly singular Hermitian metrics for the pullback bundle $\pi^{*}_{X} V \rightarrow U \times X$ defined by $\left\{K^{-1}_t\right\}_{t \in U}$ on the fiber $\left(\pi^{*}_{X} V\right)_{(t,z)} \cong V_z$ is positively curved. Otherwise, if either $\Xi_{\delta,\eta}(h) \geq_{\mathrm{Griff}} 0$ or $$\dbar_X\left(\left(h^{[t]}\right)^{-1}\del_X h^{[t]}\right) + \left(\ricci(g) + 2\del_X\dbar_X\eta-(1+\delta)\del_X\eta\wedge\dbar_X\eta\right) \otimes \mathrm{Id}_V \geq_{\mathrm{Nak}} 0,$$ 
for each $t \in U$, then the family of possibly singular Hermitian metrics for $\pi^{*}_{X} V \rightarrow U \times X$ defined by $\left\{K^{-1}_t\right\}_{t \in U}$ on the fiber $\left(\pi^{*}_{X} V\right)_{(t,z)} \cong V_z$ is semipositively curved.
\end{thmx}

\begin{thmx}\label{family-measures-trivial-theorem}
Let $(X,g)$ be any Stein Kähler manifold, let $U \subset \C^m$ be a domain, and let $V \rightarrow X$ be a holomorphic vector bundle. Let $\left\{h^{[t]}\right\}_{t\in U}$ be a family of smooth Hermitian metrics for $V \rightarrow X$. Let $\delta > 0$ and $\eta$ be smooth function on $X$. Let $\left\{\hat{\mu}_t\right\}_{t\in U}$ be a family of compactly supported $V^{*}$-valued complex measures over $X$. For each section $f \in \Gamma(E_h)$, define the measure $\mu^{(f)}_t = \left\langle f^{[t]},\hat{\mu}_t\right\rangle$. Suppose that the section $\xi^{(\mu)}$ of $E^{*}_h$ defined by
$$f^{[t]} \mapsto \left\langle\xi^{(\mu)}_t,f^{[t]}\right\rangle := \mu^{(f)}_t(X) = \int_X \left\langle f^{[t]},\hat{\mu}_t\right\rangle$$
is holomorphic. That is, $U \ni t \mapsto \mu^{(f)}_t(X)$ is holomorphic whenever $f \in \Gamma_{\mathcal{O}}(E_h)$.

If $\Xi_{\delta,\eta}(h) >_{\mathrm{Griff}} 0$ and $$\dbar_X\left(\left(h^{[t]}\right)^{-1}\del_X h^{[t]}\right) + \left(\ricci(g) + 2\del_X\dbar_X\eta-(1+\delta)\del_X\eta\wedge\dbar_X\eta\right) \otimes \mathrm{Id}_V >_{\mathrm{Nak}} 0,$$ for each $t \in U$, then the function
$$U \ni t \mapsto \log\left(\norm{\xi^{(\mu)}}^2_{h^{[t]},*}\right)$$
is strictly plurisubharmonic or identically $-\infty$. Otherwise, if either $\Xi_{\delta,\eta}(h) \geq_{\mathrm{Griff}} 0$ or $$\dbar_X\left(\left(h^{[t]}\right)^{-1}\del_X h^{[t]}\right) + \left(\ricci(g) + 2\del_X\dbar_X\eta-(1+\delta)\del_X\eta\wedge\dbar_X\eta\right) \otimes \mathrm{Id}_V \geq_{\mathrm{Nak}} 0,$$ 
for each $t \in U$, then the function
$$U \ni t \mapsto \log\left(\norm{\xi^{(\mu)}}^2_{h^{[t]},*}\right)$$
is plurisubharmonic or identically $-\infty$.
\end{thmx}

In \citep{berndtsson2017}, Berndtsson shows how his $\log$-plurisubharmonic variation results for product domains can be generalized to domains that are subsets of product domains, but not necessarily product domains themselves. His approach consists of a reduction from the latter situation to the former situation. Although Berndtsson's technique is not quite compatible with our twisted curvature condition, we succeed nonetheless at establishing $\log$-plurisubharmonic variation results for a certain class of non-trivial families of Stein manifolds.

\begin{thmx}\label{bergman-kernel-nontrivial-theorem}
Let $Y$ be an $n$-dimensional Stein manifold. Let $\rho$ be a smooth plurisubharmonic function on $\C^m \times Y$ and let 
$$X = \{\rho(t,z) < 0\} \subset \C^m \times Y.$$ Suppose that for each $t$, the restriction $\rho^{[t]}$ of $\rho$ to 
$$X_t := \{z \in Y : (t,z) \in X\} \subset Y$$
takes values in $[-1,0)$. Let $g$ be a Kähler metric for $Y$ and let us equip $\C^m \times Y$ with the product metric induced by the Euclidean metric on $\C^m$ and the metric $g$ on $Y$. Let $V \rightarrow X$ be a holomorphic vector bundle and let $h$ be a smooth Hermitian metric for $V \rightarrow X$ such that $\dbar_t\left(h^{-1}\del_Y h\right) = 0$. Let $V^{[t]} := \restr{V}{X_t}$. If $\dbar_t\left(h^{-1}\del_t h\right) >_{\mathrm{Griff}} 0$ and
$$\Theta\left(h^{[t]}\right) + \left(\ricci(g) + \del_Y\dbar_Y\rho^{[t]}\right) \otimes \mathrm{Id}_{V^{[t]}} >_{\mathrm{Nak}} 0,$$
over $X_t$, for each $t \in \C^m$, then for any $z \in X_t$, and $\sigma \in \left(V^{[t]}_z\right)^{*}$, the function
$$(t,z) \mapsto \log\left\langle \sigma \otimes \bar{\sigma}, K_t(z,\bar{z})\right\rangle$$
is strictly plurisubharmonic or identically $-\infty$. Otherwise, if either $\dbar_t\left(h^{-1}\del_t h\right) \geq_{\mathrm{Griff}} 0$ and
$$\Theta\left(h^{[t]}\right) + \left(\ricci(g) + \del_Y\dbar_Y\rho^{[t]}\right) \otimes \mathrm{Id}_{V^{[t]}} \geq_{\mathrm{Nak}} 0,$$
over $X_t$, for each $t \in \C^m$, then for any $z \in X_t$, and $\sigma \in \left(V^{[t]}_z\right)^{*}$, then the function
$$(t,z) \mapsto \log\left\langle \sigma \otimes \bar{\sigma}, K_t(z,\bar{z})\right\rangle$$
is plurisubharmonic or identically $-\infty$.
\end{thmx}

\begin{thmx}\label{family-measures-nontrivial-theorem}
Let $Y$ be an $n$-dimensional Stein manifold. Let $\rho$ be a smooth plurisubharmonic function on $\C^m \times Y$ and let 
$$X = \{\rho(t,z) < 0\} \subset \C^m \times Y.$$ Suppose that for each $t$, the restriction $\rho^{[t]}$ of $\rho$ to 
$$X_t := \{z \in Y : (t,z) \in X\} \subset Y$$
takes values in $[-1,0)$. Let $g$ be a Kähler metric for $Y$ and let us equip $\C^m \times Y$ with the product metric induced by the Euclidean metric on $\C^m$ and the metric $g$ on $Y$. Let $V \rightarrow X$ be a holomorphic vector bundle and let $h$ be a smooth Hermitian metric for $V \rightarrow X$ such that $\dbar_t\left(h^{-1}\del_Y h\right) = 0$. Let $V^{[t]} := \restr{V}{X_t}$. Moreover, let $\left\{\hat{\mu}_t\right\}_{t\in U}$ be a family of $\left(V^{[t]}\right)^{*}$-valued complex measures over $X_t$ that are all locally supported in a compact subset of $X$. For a section $f \in \Gamma(E_h)$, define the measure $\mu^{(f)}_t = \left\langle f^{[t]},\hat{\mu}_t\right\rangle$. Suppose that the section $\xi^{(\mu)}$ of $E^{*}_h$ defined by
$$f^{[t]} \mapsto \left\langle\xi^{(\mu)}_t,f^{[t]}\right\rangle := \mu^{(f)}_t(X_t) = \int_{X_t} \left\langle f^{[t]},\hat{\mu}_t\right\rangle$$
is holomorphic. That is, $U \ni t \mapsto \mu^{(f)}_t(X_t)$ is holomorphic whenever $f \in \Gamma_{\mathcal{O}}(E_h)$.

If, for each $t \in \C^m$, $\dbar_t\left(h^{-1}\del_t h\right) >_{\mathrm{Griff}} 0$ and
$$\Theta\left(h^{[t]}\right) + \left(\ricci(g) + \del_Y\dbar_Y\rho^{[t]}\right) \otimes \mathrm{Id}_{V^{[t]}} >_{\mathrm{Nak}} 0$$ 
over $X_t$, then the function
$$U \ni t \mapsto \log\left(\norm{\xi^{(\mu)}}^2_{h^{[t]},*}\right)$$
is strictly plurisubharmonic or identically $-\infty$. Otherwise, if either $\dbar_t\left(h^{-1}\del_t h\right) \geq_{\mathrm{Griff}} 0$ and
$$\Theta\left(h^{[t]}\right) + \left(\ricci(g) + \del_Y\dbar_Y\rho^{[t]}\right) \otimes \mathrm{Id}_{V^{[t]}} \geq_{\mathrm{Nak}} 0,$$
over $X_t$, for each $t \in \C^m$, then for any $z \in X_t$, and $\sigma \in \left(V^{[t]}_z\right)^{*}$, then the function
$$U \ni t \mapsto \log\left(\norm{\xi^{(\mu)}}^2_{h^{[t]},*}\right)$$
is plurisubharmonic or identically $-\infty$.
\end{thmx}

\section{Background}\label{trivial-families-general-section}
\subsection{Definitions for Fields of Hilbert Spaces}
Suppose now that $X$ is an unbounded Stein manifold. Then $E_h$ is no longer necessarily a Hilbert bundle, but rather a field of Hilbert spaces whose fibers are not necessarily isomorphic. This means that we will need to \textit{define} Griffiths positivity and Nakano positivity for these fields of Hilbert spaces alternatively. These definitions are inspired by the analytic characterizations that are equivalent to these notions of positivity in the vector bundle case (see \citep{berndtsson-2009-annals}, \citep{demailly-ag-book}, \citep{raufi2013log}, \citep{raufi2015} and \citep{varolin2019notes}). Our definitions of sections of $E_h$ and its dual are essentially the same as the ones for the locally trivial case.\\

We will be making use of Theorem \ref{thm-nakano-positivity-smooth-bounded-stein} in our proofs of Theorems \ref{griffiths-positivity-general-stein} and \ref{nakano-positivity-general-stein}.

\begin{definition}
A section $\hat{f}$ of the fields of Hilbert spaces $E_h$ is a section of $\pi^{*}_X V \rightarrow U \times X$ such that $\restr{\hat{f}}{\{t\} \times X} \in E_t := \HH_t$ for each $t \in U$. We will write $\hat{f}^{[t]} := \restr{\hat{f}}{\{t\} \times X}$.\\

The section $\hat{f}$ of $E_h$ is said to be holomorphic if it is holomorphic in the total space. In particular, all sections are holomorphic on the fibers.
\end{definition}

\begin{definition}
Let $E^{*}_h$ denote the holomorphic Hermitian field of Hilbert spaces dual to $E_h$ -- that is the fiber $E^{*}_t$ of $E^{*}_h$ over $t \in U$ is the Hilbert space dual of $\HH_t$ with its usual fiberwise Hilbert norm
$$\norm{\xi}_{*,h^{[t]}} := \sup_{f \in \HH_t-\{0\}} \dfrac{\abs{\left\langle \restr{\xi}{\HH_t},f \right\rangle}}{\norm{f}_{h^{[t]}}}.$$
\end{definition}

\begin{definition}
A section of $E^{*}_h$ is a map $\xi : E_h \rightarrow \C$ such that $\xi_t := \restr{\xi}{\HH_t} \in \left(\HH_t\right)^{*}$. The section $\xi$ is said to be smooth (resp. holomorphic) if for each smooth (resp. holomorphic) section $\hat{f}$ of $E_h$, the function $U \ni t \mapsto \left\langle \xi_t,\hat{f}^{[t]}\right\rangle \in \C$ is smooth (resp. holomorphic).\\
\end{definition}

Now, let $F_h$ denote the field of Hilbert spaces with fiber $F_t := L^2_t$ over $t \in U$ and let $\mathcal{P}_t : L^2_t \rightarrow \HH_t$ denote the fiberwise Bergman projection.

\begin{definition} (Connections)
\begin{itemize}
    \item The Chern connection $\nabla^{F_h}$ of $F_h$ is formally defined as the following collection of operators $\nabla^{F_h}_{t_j}$ for $1 \leq j \leq m$ given by $\nabla^{F_h}_{t_j}u^{[t]} = d_{t_j}u^{[t]} -\left(\left(h^{[t]}\right)^{-1}\del_{t_j} h^{[t]}\right)u^{[t]},$
    for a section $u$ of $F_h$. The domain of each $\nabla^{F_h}_{t_j}$ consists of sections $u$ of $F_h$ such that 
$$\del_{t_j}u^{[t]} -\left(\left(h^{[t]}\right)^{-1}\del_{t_j} h^{[t]}\right)u^{[t]} \in L^2_t$$ 
for each $t \in U$.
\begin{itemize}
    \item We may also define $\nabla^{F_h}_{t_j}$, as in the vector bundle case, by the relation
$$\left(\nabla^{F_h}_{t_j}u^{[t]},v_t\right)_{h^{[t]}} := d_{t_j}(u^{[t]},v_t)_{h^{[t]}} - \left(u^{[t]},d_{t_j}v_t\right)_{h^{[t]}}$$
for any two sections $u$ and $v$ of $E_h$.
\end{itemize}

\item The $(1,0)$-part of the connection is defined by the collection of operators $\nabla^{F_h,(1,0)}_{t_j}$ mapping $u^{[t]} \in L^2_t$ to $\del_{t_j}u^{[t]} -\left(\left(h^{[t]}\right)^{-1}\del_{t_j} h^{[t]}\right)u^{[t]}$. The domain of each $\nabla^{F_h,(1,0)}_{t_j}$ consists of sections $u$ of $F_h$ such that $\del_{t_j}u^{[t]} -\left(\left(h^{[t]}\right)^{-1}\del_{t_j} h^{[t]}\right)u^{[t]} \in L^2_t$ for each $t \in U$.

\item The $(0,1)$-part of the connection is defined as the collection of $\dbar$-operators $\dbar^{F_h}_{t_j}$. The domain of each $\dbar^{F_h}_{t_j}$ consists of sections $u$ of $F_h$ such that $\dbar_{t_j} u^{[t]} \in L^2_t$ for each $t \in U$.

The corresponding connections for $E_h$ are defined as the respective Bergman projections of each connection, with the domains similarly defined. So we have $\nabla^{E_h}_{t_j} := \mathcal{P}_t \circ \nabla^{F_h}_{t_j}$, $\nabla^{E_h,(1,0)}_{t_j} := \mathcal{P}_t \circ \nabla^{F_h,(1,0)}_{t_j}$ and $\dbar^{E_h}_{t_j} := \mathcal{P}_t \circ \dbar^{F_h}_{t_j}$.
\end{itemize}
\end{definition}

We will abusively denote the $\dbar$-operators for $E_h$ and $F_h$ interchangeably. 

\begin{definition} (Curvatures)
\begin{itemize}
    \item The curvature $\Theta^{F_h}$ of $F_h$ is the $(1,1)$-form of endomorphisms
    $$\Theta^{F_h} = \sum_{1 \leq j,k \leq m} \Theta^{F_h}_{t_j \bar{t}_k} dt_j \wedge d\bar{t}_k,$$
    where the multiplier coefficients $\Theta^{F_h}_{t_j \bar{t}_k}$ are endomorphisms of $V \rightarrow X$ defined on $X$ by $\Theta^{F_h}_{t_j \bar{t}_k} = \dbar_{t_k}\left(\left(h^{[t]}\right)^{-1}\del_{t_j} h^{[t]}\right)$. The domain of each $\Theta^{F_h}_{t_j \bar{t}_k}$ consists of sections $u$ of $F_h$ such that $$\dbar_{t_k}\left(\left(h^{[t]}\right)^{-1}\del_{t_j} h^{[t]}\right) u^{[t]} \in L^2_t$$ for each $t \in U$.
    \item The curvature of $\Theta^{E_h}$ of $E_h$ is the $(1,1)$-form of endomorphisms $\Theta^{E_h}_{t_j\bar{t}_k}$ of $V \rightarrow X$ defined as by the relation
    $$\left(\Theta^{F_h}_{t_j \bar{t}_k}u^{[t]},v^{[t]}\right)_{h^{[t]}} = \left(\mathcal{P}^{\perp}_t\left(\nabla^{F_h}_{t_j}u^{[t]}\right),\mathcal{P}^{\perp}_t\left(\nabla^{F_h}_{t_k}v^{[t]}\right)\right)_{h^{[t]}} + \left(\Theta^{E_h}_{t_j \bar{t}_k}u^{[t]},v^{[t]}\right)_{h^{[t]}}$$
    for any two sections $u$ and $v$ of $E_h$. The domain of each endomorphism $\Theta^{E_h}_{t_j \bar{t}_k}$ consists of sections $u$ of $E_h$ such that $\Theta^{E_h}_{t_j \bar{t}_k} u^{[t]} \in L^2_t$ for each $t \in U$.
\end{itemize}
\end{definition}

\begin{definition}\label{griffiths-positivity-definition-general}(Griffiths positivity)
    \item The holomorphic Hermitian field of Hilbert spaces $\left(E_h,(\cdot,\cdot)_{h^{[t]}}\right)$ is said to be Griffiths semipositive (resp. positive) if the function $$U \ni t \mapsto \log\left(\norm{\xi}^2_{*,h^{[t]}}\right)$$
    is (strictly) plurisubharmonic or identically $-\infty$ for every holomorphic section $\xi$ of $E^{*}_h$.
\end{definition}

It is well-known that this definition is equivalent to the usual definition of Griffiths positivity when $E_h$ is bona fide holomorphic Hilbert bundle.\\

Now, consider the $(m-1,m-1)$-form
$$T_u = \sum_{1 \leq j,k \leq m} \left(u^{[t]}_j,u^{[t]}_k\right)_{h^{[t]}} \widehat{dt_j \wedge d\bar{t}_k},$$
for an $m$-tuple $u = (u_1, \cdots, u_m)$ of holomorphic sections of $E_h$. Here,
$$\widehat{dt_j \wedge d\bar{t}_k} = c_n dt_1 \wedge \cdots \wedge dt_{j-1} \wedge dt_{j+1} \wedge \cdots \wedge dt_m \wedge d\bar{t}_1 \wedge \cdots \wedge d\bar{t}_{k-1} \wedge d\bar{t}_{k+1} \wedge d\bar{t}_m,$$
where $c_n$ is a unimodular constant chosen so that $\widehat{dt_j \wedge d\bar{t}_k}$ is a positive form.
\newpage
\begin{definition}\label{nakano-positivity-general-definition} (Nakano positivity)
    \begin{itemize}
        \item The holomorphic Hermitian field of Hilbert spaces $\left(E_h,(\cdot,\cdot)_{h^{[t]}}\right)$ is said to be Nakano positive (resp. semipositive) at $t \in U$ if 
        $$\exists c_0 > 0 \text{ (resp. } = 0\text{)}: \del_U\dbar_U\left(-T_u\right) \geq c_0\sum_{k = 1}^m \norm{u^{[t]}_k}^2_{h^{[t]}}dV(t)$$
        for any $m$-tuple $(u_1, \cdots, u_m)$ of holomorphic sections of $E_h$ belonging to the domains of $\nabla^{E_h,(1,0)}_{t_j}$ and $\Theta^{E_h}_{t_j \bar{t}_k}$ and such that $\nabla^{E_h,(1,0)}_{t_j}u^{[t]}_j = 0$ at $t$, for all $1 \leq j,k \leq m$.\\
        \item The holomorphic Hermitian field of Hilbert spaces $\left(E_h,(\cdot,\cdot)_{h^{[t]}}\right)$ is said to be Nakano (semi)positive if it is Nakano (semi)positive at every $t \in U$.
    \end{itemize}
\end{definition}

This definition is equivalent to the usual definition of Nakano positivity when $E_h$ is bona fide holomorphic Hilbert bundle (see, for instance, \citep{varolin2019notes}).

\subsection{Bergman kernels and their properties}

The following proposition, known as Bergman's inequality, captures the local uniform boundedness of the point evaluation map which defines complex-geometric reproducing kernel Hilbert spaces ($\mathbb{C}$-RKHS); also known as Bergman spaces.

\begin{proposition}\label{prop-bergman-regularity}
Let $X$ be a complex manifold with Borel measure $\mu$ and let $V \rightarrow X$ be a holomorphic vector bundle with singular Hermitian metric $h$. Assume that
\begin{enumerate}
    \item the Borel measure $\mu$ is absolutely continuous with respect to Lebesgue measure, and its local Radon-Nikodym derivatives are locally bounded below,
    \item the metric $h$ is locally bounded below, i.e., for every local frame $e_1, \cdots, e_r$ of $V \rightarrow X$, there exists a constant $c_0 > 0$ such that $\abs{\sum_{i = 1}^n v_i e_i}^2_h \geq c_0 \sum_{i = 1}^n \abs{v}^2_i$.
\end{enumerate}
Then for every set $K$ with compact closure in $X$, there exists a constant $C_K$ such that if $s \in \mathcal{H}^2(\mu,h)$, then
$$\sup_{K}\abs{s}^2_h \leq C_K\int_X \abs{s}^2_h d\mu.$$
\end{proposition}

In the sequel, given a complex manifold $X$ with bordel measure $\mu$ and a holomorphic vector bundle $V \rightarrow X$ with Hermitian metric $h$ satisfying Proposition \ref{prop-bergman-regularity}, we refer to $(\mu,h)$ as a $\mathbb{C}$-RKHS structure.\\

Let $\mathcal{H}^2(\mu,h)$ denote the space of holomorphic sections of $V \rightarrow X$ whose square-norm, with respect to the metric $h$, is integrable with respect to the measure $\mu$. The Bergman kernel of $\mathcal{H}^2(\mu,h)$ is given by the series
$$K = \sum_{j \geq 0} s_j \otimes \bar{s}_j$$
which is locally uniformly convergent, but not necessarily convergent in $\mathcal{H}^2(\mu,h) \otimes \left(\mathcal{H}^2(\mu,h)\right)^{\dagger}$. Here, $\{s_j\}_{j \geq 0}$ denotes any orthonormal Riesz basis of $\mathcal{H}^2(\mu,h)$ and $M^{\dagger}$ is the complex manifold with the complex conjugate structure of $M$. This series expansion is useful in proving the following extremal characterization of the Bergman kernel.

\begin{theorem}\label{theorem-extermal-bergman}
Let $\mathcal{H}^2(\mu,h)$ be a Bergman space. The Bergman kernel $K \in \holosections\left(X \times X^{\dagger},V\boxtimes V^{\dagger}\right)$ of $\mathcal{H}^2(\mu,h)$ is uniquely determined by the extremal property
\begin{equation}
    \forall x \in X, \sigma \in V^{*}_x : \left\langle\sigma \otimes \bar{\sigma},K(x,\bar{x})\right\rangle = \sup_{u \in \mathcal{H}^2(\mu,h)-\{0\}} \dfrac{\abs{\left\langle\sigma,u(x)\right\rangle}^2}{\norm{u}^2_{L^2(\mu,h)}}
\end{equation}

Moreover, the supremum is a maximum. That is for each $\sigma \in V^{*}$, there exists a section $u \in \mathcal{H}^2(\mu,h)$ that is unique up to a unimodular constant factor, such that $\norm{u}_{L^2(\mu,h)} = 1$ and $\left\langle \sigma \otimes \bar{\sigma},K(\pi\sigma,\overbar{\pi\sigma})\right\rangle = \abs{\left\langle\sigma,u(\pi\sigma)\right\rangle}^2$, for all $\sigma \in V^{*}$, where $\pi : V^{*} \rightarrow X$ denotes the bundle projection.
\end{theorem}
The following three propositions generalize \citep[Lemmas 3.1, 3.2, and 3.3]{berndtsson-subharmonicity}. We prove the first one. The next two essentially follow from special cases of the generalization of Ramadanov's theorem in \citep{pasternak-winiarski-wojcicki-weighted-ramadanov}. They can be shown, for instance, following the methods of proof found in \citep{pasternak-winiarski-wojcicki-weighted-ramadanov}.

\begin{proposition}\label{prop-bergman-approx}
Let $\Omega_0$ and $\Omega_1$ be bounded domains in a complex manifold $X$ such that $\Omega_0 \Subset \Omega_1$. Let $V \rightarrow \Omega_1$ be a holomorphic vector bundle and let $\mu$ be a Borel measure for $\Omega_1$. Let $\{h_j\}_{j \geq 0}$ be a sequence of Hermitian metrics for $V \rightarrow \Omega_1$ such that $(\mu,h_j)$ is a $\C$-RKHS structure for each $j$. Assume further that $\restr{h_j}{\overbar{\Omega}_0} = h$ for some metric $h$ for $\restr{V}{\overbar{\Omega}_0} \rightarrow \overbar{\Omega}_0$, and that $h_j \searrow 0$ almost everywhere in $\Omega_1 - \Omega_0$. Assume that $\mathcal{H}^2(\Omega_1,h)$ is dense in $\mathcal{H}^2(\Omega_0,h)$. Fix a point $z \in \Omega_0$. Let $K_j$ be the Bergman kernel for $\mathcal{H}^2(\Omega_1,h_j)$, and let $K$ be the Bergman kernel for $\mathcal{H}^2(\Omega_0,h)$. Then, denoting by $\iota : \Omega_0 \hookrightarrow \Omega_1$ the inclusion map,
$$\forall \sigma \in V^{*}_{\iota z} : \left\langle \left(\iota^{-1}\right)^{*}\left(\sigma \otimes \bar{\sigma}\right), K_j\left(\iota z,\overbar{\iota z}\right)\right\rangle \nearrow \left\langle \sigma\otimes\bar{\sigma},K(z,\bar{z}) \right\rangle.$$
\end{proposition}

\begin{proof}
Let $\iota : \Omega_0 \hookrightarrow \Omega_1$ be the inclusion map, and let $\sigma \in V^{*}_{\iota z}$. By Theorem \ref{theorem-extermal-bergman}, the sequence $\left\{\left\langle \left(\iota^{-1}\right)^{*}\left(\sigma \otimes \bar{\sigma}\right), K_j\left(\iota z,\overbar{\iota z}\right)\right\rangle\right\}_{j \geq 0}$ is an increasing sequence and
$$\left\langle \left(\iota^{-1}\right)^{*}\left(\sigma \otimes \bar{\sigma}\right), K_j\left(\iota z,\overbar{\iota z}\right)\right\rangle \leq \left\langle \sigma\otimes\bar{\sigma},K(z,\bar{z}) \right\rangle$$
for each $j$. Since
$$K_j(\iota z,\bar{\iota z}) = \int_{\Omega_1} \abs{K_j(\iota z, y)}^2_{h_j} d\mu(y)$$
by the uniqueness of Bergman kernels, it follows that $\{K_j\}_{j \geq 0}$ has uniformly bounded norm norm in $\mathcal{H}^2\left(\Omega_1,h_j\right)$. Therefore, the sequence $\{K_j\}_{j \geq 0}$ has a weakly convergent subsequence in $\mathcal{H}^2\left(\Omega_0,h\right)$. Let $\mathfrak{K}$ be the limit of the weakly convergent subsequence. If $f$ is in $\mathcal{H}^2\left(\Omega_1,h\right)$, then by the Cauchy-Schwarz inequality
$$\abs{\int_{\Omega_1 - \Omega_0} \left\langle f(y),K_j(x,y)\right\rangle_{h_j}d\mu(y)} \leq \norm{K_j}^2_{L^2\left(\Omega_1,h_j\right)}\int_{\Omega_1-\Omega_0}\abs{f(y)}^2_{h_j}d\mu(y),$$
and the latter converges to $0$ as $j \rightarrow +\infty$. Therefore, any weak limit $\mathfrak{K}$ satisfies
$$f(x) = \int_{\Omega_0} \left\langle f(y),\mathfrak{K}(x,\bar{y})\right\rangle_h d\mu(y)$$
for any $f \in \mathcal{H}^2\left(\Omega_1,h\right)$, and since $\mathcal{H}^2(\Omega_1,h)$ is dense in $\mathcal{H}^2(\Omega_0,h)$, the same reproducing property holds for any $f \in \mathcal{H}^2(\Omega_0,h)$. Since $\mathfrak{K}$ is holomorphic, $\mathfrak{K} = K$ by uniqueness and the limit is uniform on compact subsets of $\Omega_0$. In particular, for each $z \in \Omega_0$, $K_j(\iota z,\bar{\iota z})$ converges to $K(z,\bar{z})$ as $j \rightarrow +\infty$, which implies the desired result.
\end{proof}

\begin{proposition}\label{ramadanov-manifold-bergman-union-domains}
Let $\Omega$ be a bounded domain, with Borel measure $\mu$, in a complex manifold $X$, and let $V \rightarrow \Omega$ be a holomorphic vector bundle with Hermitian metric $h$ such that $h$ is locally bounded below and $(\mu,h)$ is a $\C$-RKHS structure. Let $\{\Omega_j\}_{j \geq 1}$ be an increasing family of subdomains with union equal to $\Omega$. Let $z$ be a fixed point in $\Omega_0$ and let $K_j$ and $K$ be the Bergman kernels for $\mathcal{H}^2\left(\Omega_j,\restr{\mu}{\Omega_j},h\right)$ and $\mathcal{H}^2(\Omega,\mu,h)$ respectively. Then, denoting by $\iota_j : \Omega_0 \hookrightarrow \Omega_j$ and $\iota : \Omega_0 \hookrightarrow \Omega$ the inclusion maps,
$$\forall \sigma \in V^{*}_{\iota_0 z} = V^{*}_{\iota_j z} : \left\langle\left(\iota_j^{-1}\right)^{*}\left(\sigma\otimes\bar{\sigma}\right),K_j\left(\iota_j z,\overbar{\iota_j z}\right)\right\rangle \searrow \left\langle \left(\iota^{-1}_0\right)^{*}\left(\sigma\otimes\sigma\right),K\left(\iota_0 z,\overbar{\iota_0 z}\right)\right\rangle.$$
\end{proposition}

\begin{proposition}\label{ramadanov-manifold-bergman}
Let $\Omega$ be a bounded domain, with Borel measure $\mu$, in a complex manifold $X$, and let $V \rightarrow \Omega$ be a holomorphic vector bundle. Suppose that $\{h_j\}_{j \geq 1}$ is a sequence of Hermitian metrics that are increasing to a metric $h$. Suppose that for each $j$, $(\mu,h_j)$ is a $\C$-RKHS structure, and that $(\mu,h)$ is also a $\C$-RKHS structure. Let $z$ be a fixed point in $\Omega$ and let $K_j$ and $K$ be the Bergman kernel for $\mathcal{H}^2(\mu,h_j)$ and $\mathcal{H}^2(\mu,h)$ respectively. Then,
$$\forall \sigma \in V^{*}_z : \left\langle \sigma \otimes \bar{\sigma}, K_j(z,\bar{z})\right\rangle \searrow \left\langle \sigma \otimes \bar{\sigma}, K(z,\bar{z})\right\rangle.$$
\end{proposition}

\begin{proposition}\label{upper-semilemma-bergman}
Let $X = \{(t, z) \in \C^m \times Y:  \rho(t, z) < 0\}$ where $Y$ is a Stein manifold and $\rho$ is plurisubharmonic near the closure of $X$. Let $V \rightarrow X$ be a holomorphic vector bundle equipped with a Hermitian metric $h$ that is smooth and locally bounded below near the closure of $X$. Let $V^{[t]} := \restr{V}{X_t}$. Then, for fixed $z$ and $\sigma \in \left(V^{[t]}_z\right)^{*}$, the function $t \mapsto \left\langle \sigma \otimes \bar{\sigma}, K_t(z,\bar{z})\right\rangle$ is upper semicontinuous.
\end{proposition}

\begin{proof}
Consider a point $t$ and let $s$ be a nearby point tending to $t$. We may choose $\varepsilon > 0$ so that all the fibers $X_s$ are contained in the open set $V_{\varepsilon} := \{(t,z) \in \C^m \times Y : \rho(t,z) < \varepsilon\}$. Note that any compact subset of $X_t$ is contained in all $X_s$ for $s$ sufficiently close to $t$. Let $K_s(\cdot,\bar{z})$ denote the Bergman kernel of $X_s$ for a fixed point $z$. Let $\sigma \in \left(V^{[s]}_z\right)^{*}$. Since the domains $X_s$ all contain a fixed open neighborhood of $z$, the $L^2$-norms of $\left\langle \sigma \otimes \bar{\sigma}, K_s(z,\bar{z})\right\rangle$ are bounded. Therefore, any sequence of $\left\langle \sigma \otimes \bar{\sigma}, K_s(z,\bar{z})\right\rangle$ has a subsequence that is weakly convergent on any compact subset of $X_t$. The $L^2$-norm of any weak limit $\mathfrak{K}$ cannot exceed the $\liminf$ of the $L^2$-norms of $\left\langle \sigma \otimes \bar{\sigma}, K_s(z,\bar{z})\right\rangle$ over $X_s$. By Theorem \ref{theorem-extermal-bergman}, it follows that for any $\sigma \in \left(V^{[s]}_z\right)^{*}$ and $\tau \in \left(V^{[t]}_z\right)^{*}$
$$\limsup_{s \rightarrow t} \left\langle \sigma \otimes \bar{\sigma}, K_s(z,\bar{z})\right\rangle \leq \left\langle \tau \otimes \bar{\tau}, K_t(z,\bar{z})\right\rangle,$$
which completes the proof.
\end{proof}

\subsection{Runge approximation theorem}
We state here a Runge approximation theorem for holomorphic sections of a vector bundle over a complex manifold which can be proved using $L^2$-estimates for the $\dbar$-operator.

\begin{proposition}\label{runge-l2-approx-prop}
Let $Y$ be a Stein manifold, and let $\Omega_0$ and $\Omega_1$ be smoothly bounded pseudoconvex domains in $Y$ with $\Omega_0$ relatively compact $\Omega_1$. Assume there is a smooth plurisubharmonic function $\rho$ in $\overbar{\Omega}_1$ such that $\Omega_0 = \{z \in \Omega_1 : \rho(z) < 0\}$. Let $V \rightarrow \Omega_1$ be a holomorphic vector bundle, and let $h$ be a Hermitian metric for $V \rightarrow \Omega_1$. Then holomorphic sections in $L^2(\Omega_1,h)$ are dense in the space of holomorphic sections in $L^2\left(\Omega_0,\restr{h}{\Omega_0}\right)$.
\end{proposition}

See \citep{demailly-estimation-l2}, for example, for a proof.

\section{Griffiths positivity for trivial families of unbounded Stein manifolds}\label{trivial-families-general-griffiths-positivity}

We now proceed to the proof of Theorem \ref{griffiths-positivity-general-stein}.

\begin{proof}
Since the result is local, we may assume that $U$ is bounded. Let $X$ be an unbounded Stein manifold. Let us denote
$$\norm{\xi}^2_{*,h^{[t]},X} := \sup_{f^{[t]} \in \HH_t-\{0\}} \dfrac{\abs{\langle \xi_t,f^{[t]} \rangle}^2}{\norm{f^{[t]}}^2_{h^{[t]},X}},$$
for a holomorphic section $\xi$ of $E^{*}_h$.\\

Let $\xi$ be an arbitrary holomorphic section of $E_h$. If $\xi \equiv 0$, then $\norm{\xi}^2_{*,h^{[t]},X} = 0$ and so the function $U \ni t \mapsto \log\left(\norm{\xi}^2_{*,h^{[t]},X}\right)$ is identically $-\infty$. We thus assume for the remainder of the proof that $\xi$ is not identically $0$.\\

Our goal is to shows that $U \ni t \mapsto \norm{\xi}^2_{*,h^{[t]},X}$ is strictly plurisubharmonic or plurisubharmonic depending on the twisted curvature assumption.\\

Since $X$ is a Stein manifold, we may express $X$ as $X = \bigcup_{j \geq 1} X_j$ where $\{X_j\}_{j \geq 1}$ is an increasing sequence of relatively compact  such that for each $j$, $X_j$ has compact closure in $X_{j+1}$. Let $\mathcal{P}_{U \times X}$ denote the Bergman projection of $L^2(U \times X,h)$ onto $\HH(U \times X,h)$ and let $\mathcal{P}^{\perp}_{U \times X}$ denote its orthogonal complement. For each $j$, let $\chi_j : X \rightarrow [0,1]$ be a smooth function supported on $X_{j+1}$ that is identically $1$ on $\overbar{X}_{j}$. In addition, let $E_{(j+1),h}$ be the bundle whose fiber over $t \in U$ is $\HH\left(X_{j+1},h^{[t]}\right) =: \HH_{t,(j+1)}$ and let $\hat{\chi}_j := \chi_j \circ \pi_X : U \times X \rightarrow X \rightarrow [0,1]$.

Since $\{t\} \times X \cong X$ and $\{t\} \times X_j \cong X_j$ for each fixed $t \in U$ and for each $j$, we will abusively denote $\HH\left(\{t\} \times X,h^{[t]}\right)$ and $\HH\left(\{t\} \times X_{j+1},h^{[t]}\right)$ by $\HH_t$ and $\HH_{t,(j+1)}$, respectively. Thus, while we take the fibers of $E_h$ and $E_{(j+1),h}$ to be $\HH_t$ and $\HH_{t,(j+1)}$, respectively, we are really thinking of $\HH\left(\{t\} \times X,h^{[t]}\right)$ and $\HH\left(\{t\} \times X_{j+1},h^{[t]}\right)$ respectively.

For any section $f \in \holosections\left(E_{(j+1),h}\right)$,
$\mathcal{P}_{U \times X}\left(\hat{\chi}_j f\right) \in \holosections\left(U \times X,\pi^{*}_X V\right) \cap L^2\left(U \times X,h\right).$ 
Therefore, the formula
$$\forall f \in \holosections(E_{(j+1),h}) : \left\langle \xi^{(j)}_t,f^{[t]}\right\rangle := \left\langle \xi_t,\restr{\mathcal{P}_{U \times X}\left(\hat{\chi}_j f\right)}{\{t\} \times X}\right\rangle,$$
defines a holomorphic section $\xi^{(j)} \in E^{*}_{(j+1),h}$.

Let us denote the dual square-norm of $\xi^{(j)}_t$ over $\left(\HH_{t,(j+1)}\right)^{*}$ by $\norm{\xi^{(j)}}^2_{*,h^{[t]},X_{j+1}}$ and let $f \in \Gamma_{\mathcal{O}}\left(E_h\right)$ so that $f^{[t]} \in \HH_t-\{0\}$. Then, we have the following estimates.
\begin{align*}
\norm{\xi^{(j)}}^2_{*,h^{[t]},X_{j+1}} &\geq \dfrac{\abs{\left\langle\xi_t,\restr{\mathcal{P}_{U \times X}\left(\hat{\chi}_j f\right)}{\{t\} \times X}\right\rangle}^2}{\norm{\restr{\mathcal{P}_{U \times X}\left(\hat{\chi}_j f\right)}{\{t\} \times X}}^2_{h^{[t]},X_{j+1}}}\\
&\geq \dfrac{\abs{\left\langle \xi_t,\restr{\mathcal{P}_{U \times X}\left(\left(\hat{\chi}_j-1\right) f\right)}{\{t\} \times X} + f^{[t]}\right\rangle}^2}{\norm{f^{[t]}}^2_{h^{[t]},X}}\\
&= \dfrac{\abs{\left\langle \xi_t,f^{[t]}\right\rangle + \left\langle\xi_t,\restr{\mathcal{P}_{U \times X}\left(\left(\hat{\chi}_j-1\right) f\right)}{\{t\} \times X}\right\rangle}^2}{\norm{f^{[t]}}^2_{h^{[t]},X}}.
\end{align*}

Now
\begin{align*}
&\abs{\left\langle \xi_t,f^{[t]}\right\rangle + \left\langle\xi_t,\restr{\mathcal{P}_{U \times X}\left(\left(\hat{\chi}_j-1\right) f\right)}{\{t\} \times X}\right\rangle}^2\\
&= \abs{\left\langle \xi_t,f^{[t]}\right\rangle}^2 + 2\mathrm{Re}\left[\left\langle \xi_t,f^{[t]}\right\rangle\overbar{\left\langle\xi_t,(\chi_j-1)f^{[t]} -\restr{\mathcal{P}^{\perp}_{U \times X}\left(\hat{\chi}_j f\right)}{\{t\} \times X}\right\rangle}\right]\\
&\, \, \, \, \, \, \, \, \, \, \, \, \, \, \, \, \, \, \, \, \, \, \, \, \, \, \, \, \, \, \, \, \, \, \, + \abs{\left\langle\xi_t,\restr{\mathcal{P}_{U \times X}\left(\left(\hat{\chi}_j-1\right) f\right)}{\{t\} \times X}\right\rangle}^2\\
&\leq \abs{\left\langle \xi_t,f^{[t]}\right\rangle}^2 + 2\abs{\left\langle \xi_t,f^{[t]}\right\rangle}\abs{\left\langle\xi_t,\restr{\mathcal{P}_{U \times X}\left(\left(\hat{\chi}_j-1\right) f\right)}{\{t\} \times X}\right\rangle}\\
&\, \, \, \, \, \, \, \, \, \, \, \, \, \, \, \, \, \, \, \, \, \, \, \, \, \, \, \, \, \, \, \, \, \, \, + \abs{\left\langle\xi_t,\restr{\mathcal{P}_{U \times X}\left(\left(\hat{\chi}_j-1\right) f\right)}{\{t\} \times X}\right\rangle}^2.
\end{align*}

Since $\xi_t$ is a continuous linear functional, there exists constants $C_1 > 0$ and $C_2 > 0$ such that
$$\abs{\left\langle \xi_t,f^{[t]}\right\rangle}^2 \leq C_1\norm{f^{[t]}}^2_{h^{[t]},X},$$
and
\begin{align*}
\abs{\left\langle \xi_t,\restr{\mathcal{P}_{U \times X}\left(\left(\hat{\chi}_j-1\right) f\right)}{\{t\} \times X}\right\rangle}^2 &\leq C_2\norm{\restr{\mathcal{P}_{U \times X}\left(\left(\hat{\chi}_j-1\right) f\right)}{\{t\} \times X}}^2_{h^{[t]},X}\\
&\leq C_2\norm{(\chi_j-1)f^{[t]}}^2_{h^{[t]},X}\\
&\leq C_2\norm{f^{[t]}}^2_{h^{[t]},X-\bar{X}_{j+1}}.
\end{align*}

Now since $\norm{f^{[t]}}^2_{h^{[t]},X-\bar{X}_j} \xrightarrow[j \rightarrow +\infty]{} 0,$
it follows that for any $\varepsilon > 0$,
$$\abs{\left\langle \xi_t,f^{[t]}\right\rangle + \left\langle\xi_t,(\chi_j-1)f^{[t]} -\restr{\mathcal{P}^{\perp}_{U \times X}\left(\hat{\chi}_j f\right)}{\{t\} \times X}\right\rangle}^2 < \dfrac{\abs{\left\langle\xi_t,f^{[t]}\right\rangle}^2}{\norm{f^{[t]}}^2_{h^{[t]},X}} + \varepsilon \leq \norm{\xi}^2_{*,h^{[t]},X} + \varepsilon,$$
provided that $j$ is sufficiently large. Therefore,
$\norm{\xi}^2_{*,h^{[t]},X} \leq \norm{\xi^{(j)}}^2_{*,h^{[t]},X_{j+1}}.$\\

The same exact argument with $X_{j+2}$ replacing $X$ shows that the sequence of dual squared norms $\left\{\norm{\xi^{(j)}}^2_{*,h^{[t]},X_{j+1}}\right\}_{j\geq 1}$ is decreasing. Moreover, this sequence is bounded below by $\norm{\xi}^2_{*,h^{[t]},X}$.\\

We now need to show that $\norm{\xi}^2_{*,h^{[t]},X}$ is indeed the limit of $\left\{\norm{\xi^{(j)}}^2_{*,h^{[t]},X_{j+1}}\right\}_{j\geq 1}$. In particular, all we need to show is that
$$\lim_{j \rightarrow +\infty} \norm{\xi^{(j)}}^2_{*,h^{[t]},X_{j+1}} \leq \norm{\xi}^2_{*,h^{[t]},X}.$$

By the definition of $\norm{\xi^{(j)}}^2_{*,h^{[t]},X_{j+1}}$, for each $j$, there exists $f_j \in \holosections(E_{(j+1),h})$ such that $$\norm{\xi^{(j)}}^2_{*,h^{[t]},X_{j+1}} = \abs{\left\langle\xi^{(j)},f^{[t]}_j\right\rangle}^2 \text{ and  } \norm{f^{[t]}_j}^2_{h^{[t]},X_{j+1}} = 1.$$ 
Extend $f_j$ by $0$ on $U \times (X-X_{j+1})$ and let $\tilde{f}_j$ be the extension of $f_j$. Then $\tilde{f}_j \in L^2(U \times X,h)$ and $\tilde{f}_j$ converges to some $\tilde{f}$ in $L^2(U \times X,h)$. In fact, $\tilde{f} \in \HH(U \times X,h)$. But then,
\begin{align*}
\norm{\xi^{(j)}}^2_{*,h^{[t]},X_{j+1}} = \abs{\left\langle \xi^{(j)}_t,\tilde{f}^{[t]}_j\right\rangle}^2 &= \abs{\left\langle \xi^{(j)}_t,\tilde{f}^{[t]}\right\rangle + \left\langle \xi^{(j)}_t,\tilde{f}^{[t]}_j-\tilde{f}^{[t]}\right\rangle}^2\\
&\leq \abs{\left\langle \xi^{(j)}_t,\tilde{f}^{[t]}\right\rangle}^2 + 2\abs{\left\langle \xi^{(j)}_t,\tilde{f}^{[t]}\right\rangle}\abs{\left\langle \xi^{(j)}_t,\tilde{f}^{[t]}_j-\tilde{f}^{[t]}\right\rangle}\\
&\ \ \ \ \ \ \ \ \ \ \ \ \ \ \ \ \, \ \ \ \ + \abs{\left\langle \xi^{(j)}_t,\tilde{f}^{[t]}_j-\tilde{f}^{[t]}\right\rangle}^2.
\end{align*}

Clearly, $\abs{\left\langle \xi^{(j)}_t,\tilde{f}^{[t]}_j-\tilde{f}^{[t]}\right\rangle}^2$ converges to $0$ as $j \rightarrow +\infty$ by continuity since $\norm{\tilde{f}^{[t]}_j-\tilde{f}^{[t]}}^2_{h^{[t]},X}$ converges to $0$ as $j \rightarrow +\infty$ by $L^2$-convergence. On the other hand,
$$\left\langle \xi^{(j)},\tilde{f}^{[t]}\right\rangle = \left\langle \xi_t,\tilde{f}^{[t]}\right\rangle + \left\langle\xi_t,\restr{\mathcal{P}_{U \times X}\left(\left(\hat{\chi}_j-1\right)f\right)}{\{t\} \times X}\right\rangle,$$
and so arguing as we previously did, we can see that
$$\lim_{j \rightarrow +\infty} \abs{\left\langle\xi^{(j)},\tilde{f}^{[t]}\right\rangle}^2 \leq \norm{\xi}^2_{*,h^{[t]},X}.$$ 
Altogether, we conclude that
$\lim_{j \rightarrow +\infty} \norm{\xi^{(j)}}^2_{*,h^{[t]},X_{j+1}} \leq \norm{\xi}^2_{*,h^{[t]},X}.$\\

By Theorem \ref{thm-nakano-positivity-smooth-bounded-stein} if $\Xi_{\delta,\eta}(h) >_{\mathrm{Griff}} 0$ and
$$\dbar_X\left(\left(h^{[t]}\right)^{-1}\del_X h^{[t]}\right) + \left(\ricci(g) + 2\del_X\dbar_X\eta-(1+\delta)\del_X\eta\wedge\dbar_X\eta\right) \otimes \mathrm{Id}_V >_{\mathrm{Nak}} 0,$$ 
for each $t \in U$, then each of the functions $U \ni t \mapsto \norm{\xi^{(j)}}^2_{*,h^{[t]},X_{j+1}}$ is strictly plurisubharmonic by Griffiths-positivity. Therefore, the function $U \ni t \mapsto \norm{\xi}^2_{*,h^{[t]},X}$ is strictly plurisubharmonic and the result follows in this case. Otherwise, if either $\Xi_{\delta,\eta}(h) \geq_{\mathrm{Griff}} 0$ or $$\dbar_X\left(\left(h^{[t]}\right)^{-1}\del_X h^{[t]}\right) + \left(\ricci(g) + 2\del_X\dbar_X\eta-(1+\delta)\del_X\eta\wedge\dbar_X\eta\right) \otimes \mathrm{Id}_V \geq_{\mathrm{Nak}} 0,$$
for each $t \in U$, then each of the functions $U \ni t \mapsto \norm{\xi^{(j)}}^2_{*,h^{[t]},X_{j+1}}$ is plurisubharmonic by Griffiths-positivity. Therefore, the function $U \ni t \mapsto \norm{\xi}^2_{*,h^{[t]},X}$ is plurisubharmonic and the result follows again, in this case.
\end{proof}

\subsubsection{Nakano positivity for trivial families of unbounded Stein manifolds}\label{trivial-families-general-nakano-positivity}

We now prove Theorem \ref{nakano-positivity-general-stein}.

\begin{proof}
Let $t \in U$ be arbitrary. Let $X$ be an unbounded Stein manifold. We may exhaust $X$ as $X = \bigcup_{j \geq 1} X_j$ where $\{X_j\}_{j\geq 1}$ is an increasing sequence of relatively compact  such that for each $j$, $X_j$ has compact closure in $X_{j+1}$. For each $j$, let $\chi_j : X \rightarrow [0,1]$ be a smooth function supported on $X_{j+1}$ and that is identically $1$ on $\overbar{X}_{j}$. Moreover, let $F_h$ denote the field of Hilbert spaces with fiber $L^2_t$ at $t \in U$ and let $F_{(j+1),h}$ and $E_{(j+1),h}$ denote the bundles with fibers $L^2_{t,(j+1)} =: L^2\left(X_{j+1},h^{[t]}\right)$ and $\mathcal{H}^2_{t,(j+1)} =: \HH\left(X_{j+1},h^{[t]}\right)$ at $t \in U$, respectively. Let $\mathcal{P}_t$ denote the orthogonal projection $L^2_t \rightarrow \HH_t$ and let $\mathcal{P}^{\perp}_t$ denote its orthogonal complement. Similarly, let $\mathcal{P}^{(j+1)}_t$ denote the orthogonal projection $L^2_{t,(j+1)} \rightarrow \HH_{t,(j+1)}$ and let $\mathcal{P}^{(j+1),\perp}_t$ denote its orthogonal complement. Then $$\nabla^{F_h}_{t_k} = \left(h^{[t]}\right)^{-1}\del_{t_k}h^{[t]} = \nabla^{F_{(j+1),h}}_{t_k} \text{ and } \Theta^{F_h}_{t_k\bar{t}_{\ell}} = \dbar_{t_{\ell}}\left(\left(h^{[t]}\right)^{-1}\del_{t_k}h^{[t]}\right) = \Theta^{F_{(j+1),h}}_{t_k\bar{t}_{\ell}}.$$

Let $u = (u_1, \cdots, u_m)$ be an $m$-tuple of holomorphic sections of $E_h$ belonging to the domains of $\nabla^{E_h,(1,0)}_{t_k}$ and $\Theta^{E_h}_{t_k\bar{t}_{\ell}}$ for each $1 \leq k,\ell \leq m$ such that $\nabla^{E_h,(1,0)}_{t_k} u^{[t]}_k = 0$ at $t$, for each $k$. Note that:
\begin{align*}
T_u &= \sum_{1 \leq k,\ell \leq m} \left(u^{[t]}_k,u^{[t]}_{\ell}\right)_{h^{[t]}} \widehat{dt_k \wedge d\bar{t}_{\ell}}\\
&= \sum_{1 \leq k,\ell \leq m} \left(\chi^2_j u^{[t]}_k,u^{[t]}_{\ell}\right)_{h^{[t]}} \widehat{dt_k \wedge d\bar{t}_{\ell}} + \sum_{1 \leq k,\ell \leq m} \left((1-\chi^2_j)u^{[t]}_k,u^{[t]}_{\ell}\right)_{h^{[t]}} \widehat{dt_k \wedge d\bar{t}_{\ell}}.
\end{align*}

Since each $u_k$ is a holomorphic section of $E_h$, it follows, by definition, that $u^{[t]}_k \in \mathcal{H}^2_t \subseteq L^2_t$ for each $t \in U$ and that each $u_k$ depends holomorphically on $t$. Let $\hat{\chi}_j := \chi_j \circ \pi_X$. Then for each $j$, $\restr{\hat{\chi}_j u_k}{\{t\} \times X} = \chi_j u^{[t]}_k \in L^2_t$ for each $t$, and each $\hat{\chi}_j u_k$ still depends holomorphically on $t$ since $\hat{\chi}_j$ is independent of $t$. Therefore, each $\hat{\chi}_j u_k$ is a holomorphic section of $F_h$. So,
\begin{align*}
    \del_U\dbar_U\left(-T_u\right) &= \sum_{1 \leq k,\ell \leq m} \left(\Theta^{F_h}_{t_k\bar{t}_{\ell}} \left(\chi^2_j u^{[t]}_k\right), u^{[t]}_{\ell}\right)_{h^{[t]}} dV(t)\\ &\ \ \ \ \ \ \ -\sum_{1 \leq k,\ell \leq m} \left(\nabla^{F_h,(1,0)}_{t_k} \left(\chi^2_j u^{[t]}_k\right),\nabla^{F_h,(1,0)}_{t_{\ell}} u^{[t]}_{\ell}\right)_{h^{[t]}} dV(t)\\
    &\ \ \ \ \ \ \ + \sum_{1 \leq k,\ell \leq m} \left(\Theta^{F_h}_{t_k\bar{t}_{\ell}}\left((1-\chi^2_j) u^{[t]}_k\right), u^{[t]}_{\ell}\right)_{h^{[t]}} dV(t)\\
    &\ \ \ \ \ \ \ -\sum_{1 \leq k,\ell \leq m} \left(\nabla^{F_h,(1,0)}_{t_k}\left((1-\chi^2_j\right) u^{[t]}_k),\nabla^{F_h,(1,0)}_{t_{\ell}}u^{[t]}_{\ell}\right)_{h^{[t]}} dV(t).
\end{align*}
But since $\chi_j$ does not depend on $t$, we have $\left[\nabla^{F_h,(1,0)}_{t_k},\chi^2_j\right] = 0 = \left[\nabla^{F_h,(1,0)}_{t_k},1-\chi^2_j\right]$ and $\left[\Theta^{F_h,(1,0)}_{t_j \bar{t}_k},\chi^2_j\right] = 0 = \left[\Theta^{F_h,(1,0)}_{t_j \bar{t}_k},1-\chi^2_j\right]$ for each $t_k$. Therefore,
\begin{align*}
    \del_U\dbar_U\left(-T_u\right) &= \sum_{1 \leq k,\ell \leq m} \left(\Theta^{F_h}_{t_k\bar{t}_{\ell}} \chi_j u^{[t]}_k, \chi_j u^{[t]}_{\ell}\right)_{h^{[t]}} dV(t)\\ &\ \ \ \ \ \ \ -\sum_{1 \leq k,\ell \leq m} \left(\nabla^{F_h,(1,0)}_{t_k} \chi_j u^{[t]}_k,\nabla^{F_h,(1,0)}_{t_{\ell}} \chi_j u^{[t]}_{\ell}\right)_{h^{[t]}} dV(t)\\
    &\ \ \ \ \ \ \ + \sum_{1 \leq k,\ell \leq m} \left((1-\chi^2_j)\Theta^{F_h}_{t_k\bar{t}_{\ell}}u^{[t]}_k, u^{[t]}_{\ell}\right)_{h^{[t]}} dV(t)\\
    &\ \ \ \ \ \ \ -\sum_{1 \leq k,\ell \leq m} \left(\left(1-\chi^2_j\right)\nabla^{F_h,(1,0)}_{t_k}u^{[t]}_k,\nabla^{F_h,(1,0)}_{t_{\ell}}u^{[t]}_{\ell}\right)_{h^{[t]}} dV(t).\\
\end{align*}

We have the following for each $k$ and $\ell$.
\begin{align*}
&\left(\Theta^{E_h}_{t_k\bar{t}_{\ell}}\left((1-\chi^2_j)u^{[t]}_k\right), u^{[t]}_{\ell}\right)_{h^{[t]}}\\
&= \left(\Theta^{F_h}_{t_k\bar{t}_{\ell}}\left((1-\chi^2_j)u^{[t]}_k\right), u^{[t]}_{\ell}\right)_{h^{[t]}} - \left(\mathcal{P}^{\perp}_{t}\left(\nabla^{F_h}_{t_k}\left((1-\chi^2_j)u^{[t]}_k\right)\right),\mathcal{P}^{\perp}_{t}\left(\nabla^{F_h}_{t_{\ell}} u^{[t]}_{\ell}\right)\right)_{h^{[t]}}\\
&= \left((1-\chi^2_j)\Theta^{F_h}_{t_k\bar{t}_{\ell}} u^{[t]}_k,\chi_j u^{[t]}_{\ell}\right)_{h^{[t]}} - \left(\mathcal{P}^{\perp}_{t}\left((1-\chi^2_j)\nabla^{F_h}_{t_k} u^{[t]}_k\right),\mathcal{P}^{\perp}_{t}\left(\nabla^{F_h}_{t_{\ell}} u^{[t]}_{\ell}\right)\right)_{h^{[t]}}.
\end{align*}

Since $\mathcal{P}^{\perp}_t$ is an orthogonal projection, the Cauchy-Schwarz inequality yields
\begin{align*}
&\left(\mathcal{P}^{\perp}_{t}\left((1-\chi^2_j)\nabla^{F_h}_{t_k} u^{[t]}_k\right),\mathcal{P}^{\perp}_{t}\left(\nabla^{F_h}_{t_{\ell}} u^{[t]}_{\ell}\right)\right)_{h^{[t]}}\\
&\leq \norm{\mathcal{P}^{\perp}_{t}\left((1-\chi^2_j)\nabla^{F_h}_{t_k} u^{[t]}_k\right)}^2_{h^{[t]},X}\norm{\mathcal{P}^{\perp}_{t}\left(\nabla^{F_h}_{t_{\ell}} u^{[t]}_{\ell}\right)}^2_{h^{[t]},X}\\
&\leq \norm{(1-\chi^2_j)\nabla^{F_h}_{t_k} u^{[t]}_k}^2_{h^{[t]},X}\norm{\nabla^{F_h}_{t_{\ell}} u^{[t]}_{\ell}}^2_{h^{[t]},X}\\
&\leq \norm{\nabla^{F_h}_{t_k}u^{[t]}_k}^2_{h^{[t]},X-\overbar{X}_j}\norm{\nabla^{F_h}_{t_{\ell}}u^{[t]}_{\ell}}^2_{h^{[t]},X} \xrightarrow[j \rightarrow +\infty]{} 0. 
\end{align*}

Similarly, for each $k$ and $\ell,$ $$\left((1-\chi^2_j)\Theta^{F_h}_{t_k\bar{t}_{\ell}} u^{[t]}_k,\chi_j u^{[t]}_{\ell}\right)_{h^{[t]}} \leq \norm{\Theta^{F_h}_{k \ell}u^{[t]}_k}^2_{h^{[t]},X-\overbar{X}_j}\norm{u^{[t]}_{\ell}}^2_{h^{[t]},X} \xrightarrow[j \rightarrow +\infty]{} 0.$$

Therefore,
$$\sum_{1 \leq k,\ell \leq m} \left((1-\chi^2_j)\Theta^{F_h}_{t_k\bar{t}_{\ell}}u^{[t]}_k, u^{[t]}_{\ell}\right)_{h^{[t]}} dV(t) \xrightarrow[j \rightarrow +\infty]{} 0.$$

Estimating each of the terms $\left(\left(1-\chi^2_j\right)\nabla^{F_h,(1,0)}_{t_k}u^{[t]}_k,\nabla^{F_h,(1,0)}_{t_{\ell}}u^{[t]}_{\ell}\right)_{h^{[t]}}$ in the same fashion, we see that
$$\sum_{1 \leq k,\ell \leq m} \left(\left(1-\chi^2_j\right)\nabla^{F_h,(1,0)}_{t_k}u^{[t]}_k,\nabla^{F_h,(1,0)}_{t_{\ell}}u^{[t]}_{\ell}\right)_{h^{[t]}} dV(t) \xrightarrow[j \rightarrow +\infty]{} 0.$$

We now focus our attention on the first two sums in the expression of $\del_U\dbar_U(-T_u)$. In what follows, let
$$(v_1,v_2)_{[j+1],h^{[t]}} := \int_{X_{j+1}} h^{[t]}\left(v_1,v_2\right) dV_g.$$

For any two holomorphic sections $u$ and $v$ of $F_h$, we have
\begin{align*}
    &\left(\Theta^{F_h}_{t_k\bar{t}_{\ell}} u^{[t]}, v^{[t]}\right)_{h^{[t]}} - \left(\nabla^{F_h,(1,0)}_{t_k}u^{[t]},\nabla^{F_h,(1,0)}_{t_\ell}v^{[t]}\right)_{h^{[t]}}\\
    &= \left(\mathcal{P}^{\perp}_t\left(\nabla^{F}_{t_k}u^{[t]}\right),\mathcal{P}^{\perp}_t\left(\nabla^{F}_{t_\ell}v^{[t]}\right)\right)_{h^{[t]}} + \left(\Theta^{E_h}_{t_k\bar{t}_{\ell}} u^{[t]}, v^{[t]}\right)_{h^{[t]}}\\
    &\ \ \ \ \ \ \ \ \ \ \ \ \ \ \ \ \ \ \ \ \ \ \ \ \ \ \ \ \ \ \ \ \ \ \ \ \ \ \ \ \ \ \, \, \, - \left(\nabla^{F_h,(1,0)}_{t_k}u^{[t]},\nabla^{F_h,(1,0)}_{t_\ell}v^{[t]}\right)_{h^{[t]}}\\
    &= \left(\Theta^{E_h}_{t_k\bar{t}_{\ell}}u^{[t]},v^{[t]}\right)_{h^{[t]}} + \left(\mathcal{P}^{\perp}_t\left(\nabla^{F_h,(1,0)}_{t_k}u^{[t]}\right),\mathcal{P}^{\perp}_t\left(\nabla^{F_h,(1,0)}_{t_\ell}v^{[t]}\right)\right)_{h^{[t]}}\\
    &\ \ \ \ \ \ \ \ \ \ \ \ \ \ \ \ \ \ \ \ \ \ \ \ \ \ - \left(\nabla^{F_h,(1,0)}_{t_k}u^{[t]},\nabla^{F_h,(1,0)}_{t_\ell}v^{[t]}\right)_{h^{[t]}}\\
    &= \left(\Theta^{E_h}_{t_k\bar{t}_{\ell}}u^{[t]},v^{[t]}\right)_{h^{[t]}} -\left(u^{[t]},\mathcal{P}_t\left(\nabla^{F_h,(1,0)}_{t_k}v^{[t]}\right)\right)_{h^{[t]}} -\left(\mathcal{P}_t\left(\nabla^{F_h,(1,0)}_{t_k}u^{[t]}\right),v^{[t]}\right)_{h^{[t]}}\\
    &\ \ \ \ \ \ \ \ \ \ \ \ \ \ \ \ \ \ \ \ \ \ \ \ \ \ +\left(\mathcal{P}_t\left(\nabla^{F_h,(1,0)}_{t_\ell}u^{[t]}\right),\mathcal{P}_t\left(\nabla^{F_h,(1,0)}_{t_\ell}v^{[t]}\right)\right)_{h^{[t}}\\
    &= \left(\Theta^{E_h}_{t_k\bar{t}_{\ell}}u^{[t]},v^{[t]}\right)_{h^{[t]}} -\left(u^{[t]},\nabla^{E_h,(1,0)}_{t_k}v^{[t]}\right)_{h^{[t]}} -\left(\nabla^{E_h,(1,0)}_{t_k}u^{[t]},v^{[t]}\right)_{h^{[t]}}\\
    &\ \ \ \ \ \ \ \ \ \ \ \ \ \ \ \ \ \ \ \ \ \ \ \ \ \ +\left(\nabla^{E_h,(1,0)}_{t_k}u^{[t]},\nabla^{E_h,(1,0)}_{t_\ell}v^{[t]}\right)_{h^{[t}}.
\end{align*}

Now, since each $u_k$ satisfies $\nabla^{E_h,(1,0)}_{t_k} u^{[t]}_k = 0$ at $t$, we have the following at $t$.
\begin{align*}
    \nabla^{E_h,(1,0)}_{t_k}\chi_j u^{[t]}_k &= \nabla^{E_h,(1,0)}_{t_k}\left(\left(\chi_j -1\right) u^{[t]}_k\right)\\
    &= \mathcal{P}_t\left(\nabla^{F_h,(1,0)}_{t_k}\left(\left(\chi_j -1\right) u^{[t]}_k\right)\right)\\
    &= \mathcal{P}_t\left(\left(\chi_j-1\right)\nabla^{F_h,(1,0)}_{t_k}u^{[t]}_k\right)
\end{align*}

But since $\mathcal{P}_t$ is an orthogonal projection,
$$\norm{\mathcal{P}_t\left(\left(\chi_j-1\right)\nabla^{F_h,(1,0)}_{t_k}u^{[t]}_k\right)}^2_{h^{[t]}} \leq \norm{\left(\chi_j-1\right)\nabla^{F_h,(1,0)}_{t_k}u^{[t]}_k}^2_{h^{[t]}} \leq \norm{\nabla^{F_h,(1,0)}_{t_k}u^{[t]}_k}^2_{h^{[t]},X-\overbar{X}_j} \xrightarrow[j \rightarrow +\infty]{} 0,$$
and so by the Cauchy-Schwarz inequality,
$$\left(u^{[t]}_k,\nabla^{E_h,(1,0)}_{t_k}u^{[t]}_\ell\right)_{h^{[t]}}, \left(\nabla^{E_h,(1,0)}_{t_k}u^{[t]}_k,u^{[t]}_\ell\right)_{h^{[t]}}, \left(\nabla^{E_h,(1,0)}_{t_k}u^{[t]}_k,\nabla^{E_h,(1,0)}_{t_\ell}u^{[t]}_\ell\right)_{h^{[t}} \xrightarrow[j \rightarrow +\infty]{} 0.$$

On the other hand, note that each $u_k$ is a holomorphic section of $E_{(j+1),h}$ and that each $\hat{\chi}_j u_k$ is simultaneously a holomorphic section of $F_{(j+1),h}$ and a smooth section of $E_{(j+1),h}$. Since $\Theta^{F_h}_{t_k\bar{t}_{\ell}} = \Theta^{F_{(j+1),h}}_{t_k\bar{t}_{\ell}}$,
\begin{align*}
\left(\Theta^{E_h}_{t_k\bar{t}_{\ell}}\chi_j u^{[t]}_k,\chi_j u^{[t]}_{\ell}\right)_{h^{[t]}} &= \left(\Theta^{E_{(j+1),h}}_{t_k\bar{t}_{\ell}}\chi_j u^{[t]}_k,\chi_j u^{[t]}_{\ell}\right)_{[j+1],h^{[t]}}\\
&\ \ \ \ + \left(\mathcal{P}^{(j+1),\perp}_{t}\left(\nabla^{F_h}_{t_k}\chi_j u^{[t]}_k\right),\mathcal{P}^{(j+1),\perp}_{t}\left(\nabla^{F_h}_{t_{\ell}}\chi_j u^{[t]}_{\ell}\right)\right)_{[j+1],h^{[t]}}\\
& \ \ \ \ - \left(\mathcal{P}^{\perp}_{t}\left(\nabla^{F_h}_{t_k}\chi_j u^{[t]}_k\right),\mathcal{P}^{\perp}_{t}\left(\nabla^{F_h}_{t_{\ell}}\chi_j u^{[t]}_{\ell}\right)\right)_{h^{[t]}}.\\
\end{align*}

Note again that because $\nabla^{E_h,(1,0)}_{t_k} u^{[t]}_k = 0$ at $t$,
$$\mathcal{P}^{\perp}_{t}\left(\nabla^{F_h}_{t_k}\chi_j u^{[t]}_k\right) = \mathcal{P}^{\perp}_{t}\left(\nabla^{F_h,(1,0)}_{t_k}\chi_j u^{[t]}_k\right) = \nabla^{E_h,(1,0)}_{t_k}\chi_j u^{[t]}_k = \nabla^{E_h,(1,0)}_{t_k}\left(\chi_j-1\right) u^{[t]}_k,$$
and so our previous arguments imply that $\left(\mathcal{P}^{\perp}_{t}\left(\nabla^{F_h}_{t_k}\chi_j u^{[t]}_k\right),\mathcal{P}^{\perp}_{t}\left(\nabla^{F_h}_{t_{\ell}}\chi_j u^{[t]}_{\ell}\right)\right)_{h^{[t]}} \xrightarrow[j \rightarrow +\infty]{} 0.$ Additionally,
\begin{align*}
\left(\Theta^{E_{(j+1),h}}_{t_k\bar{t}_{\ell}}\chi_j u^{[t]}_k,\chi_j u^{[t]}_{\ell}\right)_{[j+1],h^{[t]}} &= \left(\Theta^{E_{(j+1),h}}_{t_k\bar{t}_{\ell}}u^{[t]}_k,u^{[t]}_{\ell}\right)_{[j+1],h^{[t]}} + \left(\Theta^{E_{(j+1),h}}_{t_k\bar{t}_{\ell}}\left(\chi_j-1\right) u^{[t]}_k, u^{[t]}_{\ell}\right)_{[j+1],h^{[t]}}\\
&\ \ \ \ + \left(\Theta^{E_{(j+1),h}}_{t_k\bar{t}_{\ell}}u^{[t]}_k,\left(\chi_j -1\right)u^{[t]}_{\ell}\right)_{[j+1],h^{[t]}}\\
&\ \ \ \ + \left(\Theta^{E_{(j+1),h}}_{t_k\bar{t}_{\ell}}\left(\chi_j-1\right) u^{[t]}_k,\left(\chi_j-1\right) u^{[t]}_{\ell}\right)_{[j+1],h^{[t]}}.
\end{align*}

The last two summands converge to $0$ as $j \rightarrow +\infty$ by the Cauchy-Schwarz inequality and the fact that
$$\norm{\left(\chi_j-1\right)u^{[t]}_k}^2_{[j+1],h^{[t]}} \leq \norm{u^{[t]}_k}^2_{h^{[t]},X_{j+1}-\overbar{X}_j} \xrightarrow[j \rightarrow +\infty]{} 0.$$

As for the second summand,
\begin{align*}
&\left(\Theta^{E_{(j+1),h}}_{t_k\bar{t}_{\ell}}\left(\chi_j-1\right)u^{[t]}_k,u^{[t]}_{\ell}\right)_{[j+1],h^{[t]}}\\
&= \left(\Theta^{F_{(j+1),h}}_{t_k\bar{t}_{\ell}}\left(\chi_j-1\right)u^{[t]}_k,u^{[t]}_{\ell}\right)_{[j+1],h^{[t]}} - \left(\mathcal{P}^{(j+1),\perp}_t\left(\nabla^{F_h}_{t_k} \left(\chi_j-1\right)u^{[t]}_k\right),\mathcal{P}^{(j+1),\perp}_t \left(\nabla^{F_h}_{t_\ell} u^{[t]}_{\ell}\right)\right)_{[j+1],h^{[t]}}\\
&= \left(\Theta^{F_h}_{t_k\bar{t}_{\ell}}\left(\chi_j-1\right)u^{[t]}_k,u^{[t]}_{\ell}\right)_{[j+1],h^{[t]}} - \left(\mathcal{P}^{(j+1),\perp}_t\left(\nabla^{F_h}_{t_k} \left(\chi_j-1\right)u^{[t]}_k\right),\mathcal{P}^{(j+1),\perp}_t \left(\nabla^{F_h}_{t_\ell} u^{[t]}_{\ell}\right)\right)_{[j+1],h^{[t]}}\\
&= \left(\left(\chi_j-1\right)\Theta^{F_h}_{t_k\bar{t}_{\ell}}u^{[t]}_k,u^{[t]}_{\ell}\right)_{[j+1],h^{[t]}} - \left(\mathcal{P}^{(j+1),\perp}_t \left(\left(\chi_j-1\right)\nabla^{F_h}_{t_k} u^{[t]}_k\right),\mathcal{P}^{(j+1),\perp}_t\left(\nabla^{F_h}_{t_\ell} u^{[t]}_{\ell}\right)\right)_{[j+1],h^{[t]}}.\\
\end{align*}

Therefore,
\begin{align*}
&\left(\Theta^{E_{(j+1),h}}_{t_k\bar{t}_{\ell}}\left(\chi_j-1\right)u^{[t]}_k,u^{[t]}_{\ell}\right)_{[j+1],h^{[t]}}\\
&\leq \norm{\left(\chi_j-1\right)\Theta^{F_h}_{t_k\bar{t}_{\ell}}u^{[t]}_k}^2_{[j+1],h^{[t]}}\norm{u^{[t]}_{\ell}}^2_{[j+1],h^{[t]}}\\ 
&\ \ \ \ \ + \norm{\mathcal{P}^{(j+1),\perp}_t\left(\left(\chi_j-1\right)\nabla^{F_h}_{t_k} u^{[t]}_k\right)}^2_{[j+1],h^{[t]}}\norm{\mathcal{P}^{(j+1),\perp}_t\nabla^{F_h}_{t_\ell}\left(u^{[t]}_{\ell}\right)}^2_{[j+1],h^{[t]}}\\
&\leq \norm{\left(\chi_j-1\right)\Theta^{F_h}_{t_k\bar{t}_{\ell}}u^{[t]}_k}^2_{[j+1],h^{[t]}}\norm{u^{[t]}_{\ell}}^2_{[j+1],h^{[t]}} + \norm{\left(\chi_j-1\right)\nabla^{F_h}_{t_k} u^{[t]}_k}^2_{[j+1],h^{[t]}}\norm{\nabla^{F_h}_{t_\ell} u^{[t]}_{\ell}}^2_{[j+1],h^{[t]}}\\
&\leq \norm{\Theta^{F_h}_{t_k\bar{t}_{\ell}}u^{[t]}_k}^2_{h^{[t]},X_{j+1}-\overbar{X}_j}\norm{u^{[t]}_{\ell}}^2_{[j+1],h^{[t]}} + \norm{\nabla^{F_h}_{t_k} u^{[t]}_k}^2_{h^{[t]},X_{j+1}-\overbar{X}_j}\norm{\nabla^{F_h}_{t_\ell} u^{[t]}_{\ell}}^2_{[j+1],h^{[t]}}
\xrightarrow[j \rightarrow +\infty]{} 0.
\end{align*}

Altogether, we have
$$\del_U\dbar_U(-T_u) = \sum_{1 \leq k,\ell \leq m} \left(\Theta^{E_{(j+1),h}}_{t_k\bar{t}_{\ell}} u^{[t]}_k, u^{[t]}_{\ell}\right)_{[j+1],h^{[t]}} dV(t) + o(j).$$

By Theorem \ref{thm-nakano-positivity-smooth-bounded-stein}, if $\Xi_{\delta,\eta}(h) >_{\mathrm{Griff}} 0$ and
$$\dbar_X\left(\left(h^{[t]}\right)^{-1}\del_X h^{[t]}\right) + \left(\ricci(g) + 2\del_X\dbar_X\eta-(1+\delta)\del_X\eta\wedge\dbar_X\eta\right) \otimes \mathrm{Id}_V >_{\mathrm{Nak}} 0,$$
for each $t \in U$, then
$$\exists c_0 > 0: \sum_{1 \leq k,\ell \leq m} \left(\Theta^{E_{(j+1),h}}_{t_k\bar{t}_{\ell}} u^{[t]}_k, u^{[t]}_{\ell}\right)_{[j+1],h^{[t]}} \geq c_0\sum_{k=1}^m \norm{ u^{[t]}_k}^2_{[j+1],h^{[t]}}.$$

For each $u_k$,
$\norm{u^{[t]}_k}^2_{[j+1],h^{[t]}} = \norm{u^{[t]}_k}^2_{h^{[t]}} - \norm{u^{[t]}_k}^2_{h^{[t]},X-\overbar{X}_{j+1}}$ and since $u^{[t]}_k \in \mathcal{H}^2_t \subseteq L^2_t$ for each $u_k$, $\norm{u^{[t]}_k}^2_{h^{[t]},X-\overbar{X}_{j+1}} \xrightarrow[j \rightarrow +\infty]{} 0$. Therefore,
$$\exists c_0 > 0: \del_U\dbar_U(-T_u) \geq c_0\sum_{k = 1}^m \norm{u^{[t]}_k}^2_{h^{[t]}}dV(t) + o(j),$$
whence $\left(E,(\cdot,\cdot)_{h^{[t]}}\right)$ is Nakano positive at $t$ by taking the limit as $j \rightarrow +\infty$. The result follows as $t \in U$ was arbitrary.

Otherwise, if either $\Xi_{\delta,\eta}(h) \geq_{\mathrm{Griff}} 0$ or $$\dbar_X\left(\left(h^{[t]}\right)^{-1}\del_X h^{[t]}\right) + \left(\ricci(g) + 2\del_X\dbar_X\eta-(1+\delta)\del_X\eta\wedge\dbar_X\eta\right) \otimes \mathrm{Id}_V \geq_{\mathrm{Nak}} 0,$$
for each $t \in U$, then again by Theorem \ref{thm-nakano-positivity-smooth-bounded-stein} and our previous observations,
$$\del_U\dbar_U(-T_u) \geq o(j),$$
whence $\left(E,(\cdot,\cdot)_{h^{[t]}}\right)$ is Nakano semipositive at $t$ by taking the limit as $j \rightarrow +\infty$. The result follows in this case as well since $t \in U$ was arbitrary.
\end{proof}

\section{Plurisubharmonic variation results for trivial families of unbounded Stein manifolds}
\subsection{Variations of Bergman kernels}
One immediate consequence of Theorem \ref{griffiths-positivity-general-stein} is Theorem \ref{bergman-kernel-trivial-theorem}, which we now prove.

\begin{proof}
For $z \in X$ and $t \in U$, define $\hat{\xi}^{(z)}_t$ by $\hat{\xi}^{(z)}_t (\hat{f}) = i^{*}_t \hat{f}(z),$
for $\hat{f} \in \holosections\left(E_h\right)$. By Bergman's inequality, $\hat{\xi}^{(z)}_t : \HH_t \rightarrow V_z$ is a bounded linear map. Now let $\sigma \in V^{*}_z$ be a non-zero vector. Then
$$\xi^{(z,\sigma)}_t := \sigma \otimes \hat{\xi}^{(z)}_t \in \left(\HH_t\right)^{*}.$$

Moreover, if $f \in \holosections\left(E_h\right)$, then the function $t \mapsto \left\langle\xi^{(z,\sigma)}_t,f^{[t]}\right\rangle$ is holomorphic. Thus $t \mapsto \xi^{(z,\sigma)}_t$ defines a holomorphic section of $E^{*}_h$ which we denote by $\xi^{(z,\sigma)}$.\\

The fiberwise dual squared norm of this section is given by
$$\norm{\xi^{(z,\sigma)}}^2_{*,h^{[t]}} = \sup_{f \in \HH_t-\{0\}} \dfrac{\abs{\left\langle \xi^{(z,\sigma)},f\right\rangle}^2}{\norm{f}^2_{h^{[t]}}} = \sup_{f \in \HH_t-\{0\}} \dfrac{\abs{\left\langle f(z),\sigma\right\rangle}^2}{\norm{f}^2_{h^{[t]}}} = \left\langle K_t(z,z),\sigma \otimes \bar{\sigma}\right\rangle.$$

By Theorem \ref{thm-nakano-positivity-smooth-bounded-stein}, if $\Xi_{\delta,\eta}(h) >_{\mathrm{Griff}} 0$ and $$\dbar_X\left(\left(h^{[t]}\right)^{-1}\del_X h^{[t]}\right) + \left(\ricci(g) + 2\del_X\dbar_X\eta-(1+\delta)\del_X\eta\wedge\dbar_X\eta\right) \otimes \mathrm{Id}_V >_{\mathrm{Nak}} 0,$$ 
for each $t \in U$, then
$$\del_U\dbar_U\log \left\langle K_t(z,z),\sigma \otimes \bar{\sigma}\right\rangle > 0.$$
Otherwise, if either $\Xi_{\delta,\eta}(h) \geq_{\mathrm{Griff}} 0$ or $$\dbar_X\left(\left(h^{[t]}\right)^{-1}\del_X h^{[t]}\right) + \left(\ricci(g) + 2\del_X\dbar_X\eta-(1+\delta)\del_X\eta\wedge\dbar_X\eta\right) \otimes \mathrm{Id}_V \geq_{\mathrm{Nak}} 0,$$
for each $t \in U$, then
$$\del_U\dbar_U\log \left\langle K_t(z,z),\sigma \otimes \bar{\sigma}\right\rangle \geq 0,$$
which completes the proof.
\end{proof}

In the Euclidean setting, this result corresponds to Berndtsson's result on the log-plurisubharmonic variation of Bergman kernels for product domains (\citep{berndtsson-subharmonicity}). In our setting, we have the following theorem.

\begin{theorem}
Let $U$ be a domain in $\C^m$ and $\Omega$ a pseudoconvex domain in $\C^n$. Let $\delta > 0$, let $\varphi \in \mathcal{C}^{\infty}\left(U \times \Omega\right)$, and let $\eta \in \mathcal{C}^{\infty}(\Omega)$.
Let $K_t$ denote the Bergman kernel of the projection of $L^2\left(\Omega,e^{-\varphi^{[t]}}\right)$ onto $\mathcal{H}^2\left(\Omega,e^{-\varphi^{[t]}}\right)$. If $\Xi_{\delta,\eta}\left(e^{-\varphi}\right) \geq 0$ \emph{(resp. $> 0$)} in $U \times \Omega$, then the function $t \mapsto \log\left(K_t(z,z)\right)$ is plurisubharmonic (resp. strictly plurisubharmonic) or identically $-\infty$.
\end{theorem}

\subsubsection{Variations of families of compactly supported measures}

Let $\{\hat{\mu}_t\}_{t \in U}$ be a family of compactly supported $V^{*}$-valued complex measures over $X$. Then for each section $f$ of $E_h$, the pairing $\mu^{(f)}_t := \left\langle f^{[t]},\hat{\mu}_t\right\rangle$
defines a compactly supported complex measure on $X$ for each $t \in U$. Now consider the mapping $\xi^{(\mu)}_t$ defined by
$$f^{[t]} \mapsto \left\langle \xi^{(\mu)}_t,f^{[t]}\right\rangle := \mu^{(f)}_t(X) =\int_X \left\langle f^{[t]},\hat{\mu}_t\right\rangle.$$

Let us represent $\hat{\mu}_t$ locally as $\hat{\mu}_t = \sum_{j = 1}^r \sigma_j \otimes \mu^{(j)}_t$ where $\sigma_1, \cdots, \sigma_r$ be a local frame for $V^{*}$ over some open subset $W \subset X$, and let $\mu^{(1)}_t, \cdots, \mu^{(r)}_t$ be complex measures for $X$ over $W$, all of which are supported on a compact subset $K$ of $X$. Let $h^{[t],*}$ denote the dual metric to $h^{[t]}$ for the dual bundle $V^{*} \rightarrow X$. Then, locally, we have the following
\begin{align*}
\abs{\int_X \left\langle f^{[t]}, \hat{\mu}_t\right\rangle} &= \abs{\int_X \sum_{j = 1}^r \left\langle f^{[t]},\sigma_j\right\rangle d\mu^{(j)}_t}\\
&\leq \int_X \sum_{j = 1}^r \abs{\left\langle f^{[t]},\sigma_j\right\rangle} d\mu^{(j)}_t\\
&\leq \int_X \abs{f}_{h^{[t]}}\left(\sum_{j = 1}^r \abs{\sigma_j}_{h^{[t],*}}\right)d\mu^{(j)}_t\\
&= \int_K \abs{f}_{h^{[t]}}\left(\sum_{j = 1}^r \abs{\sigma_j}_{h^{[t],*}}\right)d\mu^{(j)}_t\\
&\leq \sup_{K}\left(\sum_{j = 1}^r \abs{\sigma_j}_{h^{[t],*}}\right)\sup_{K}\abs{f}_{h^{[t]}}\int_{K} d\mu^{(j)}_t\\
&= \sup_{K}\left(\sum_{j = 1}^r \abs{\sigma_j}_{h^{[t],*}}\right)\mu^{(j)}_t(K)\sup_{K}\abs{f}_{h^{[t]}}.
\end{align*}
Now by Bergman's inequality, there exists a constant $C_{K} > 0$ such that $$\sup_K \abs{f}^2_{h^{[t]}} \leq C_K\int_{X} \abs{f}^2_{h^{[t]}} dV_g.$$

Therefore,
$$\sup_{f^{[t]} \in \mathcal{H}^2_t} \dfrac{\abs{\left\langle \xi^{(\mu)}_t, f^{[t]}\right\rangle}^2}{\norm{f}^2_{h^{[t]}}}$$
is bounded. Using a partition of unity, we can see that this boundedness does not depend on the choice of frame or local representative measures. So the mapping $\xi^{(\mu)}_t$ defines a smooth section $\xi^{(\mu)}$ of $E^{*}_h$. We thus have Theorem \ref{family-measures-trivial-theorem} as a consequence of Theorem \ref{griffiths-positivity-general-stein}.\\

In the Euclidean setting, Theorem \ref{family-measures-trivial-theorem} corresponds to Berndtsson's theorem on the $\log$-plurisubharmonic variation of compactly supported measures for product domains (\citep{berndtsson2018complex,berndtsson2017}). In our twisted setting, we have the following theorem.

\begin{theorem}
Let $U$ be a domain in $\C^m$ and $\Omega$ a pseudoconvex domain in $\C^n$. Let  $\delta > 0$, let $\varphi \in \mathcal{C}^{\infty}\left(U \times \Omega\right)$, and let $\eta \in \mathcal{C}^{\infty}(\Omega)$. Let $\left\{\mu_t\right\}_{t \in U}$ be a family of complex measures on $\Omega$ which are compactly supported inside $\Omega$. Define a section $\xi^{(\mu)}$ of $E^{*}_{\varphi}$ by
$$\forall t\in U, \forall F \in \HH_{\varphi(t,\cdot)}(\Omega) : \xi^{(\mu)}_t(F) = \int_{\Omega} F(z) d\mu_t(z).$$
Suppose that $\xi^{(\mu)}$ is a holomorphic section of $E^{*}_{\varphi}$. If $\Xi_{\delta,\eta}\left(e^{-\varphi}\right) \geq 0$ \emph{(resp. $> 0$)} in $U \times \Omega$, then the function 
$$U \ni t \mapsto \log\left(\norm{\xi^{(\mu)}}^2_{*,\varphi(t,\cdot)}\right)$$ 
is plurisubharmonic (resp. strictly plurisubharmonic) or identically $-\infty$.
\end{theorem}

Furthermore, letting $\hat{\mu}_t := \sigma \otimes \delta_{F(t)}$ for each $t \in U$, where $F$ is a holomorphic map from $U$ to $X$, $\delta_{F(t)}$ denotes a point-mass measure supported at $F(t)$ and $\sigma \in V^{*}_{F(t)}$, we obtain a slightly stronger result than Theorem \ref{bergman-kernel-nontrivial-theorem}. Namely, under the hypotheses of Theorem \ref{bergman-kernel-nontrivial-theorem}, the function $U \ni t \mapsto \log\left\langle \sigma \otimes \bar{\sigma},K_t\left(F(t),F(t)\right)\right\rangle$ is plurisubharmonic (or strictly so) or identically $-\infty$ for every $\sigma \in V^{*}_{F(t)}$.

\subsection{Twisted log-plurisubharmonic variation results for a class of non-trivial families of Stein manifolds}

Now let $Y$ be an $n$-dimensional Stein manifold, let $\rho$ be a smooth plurisubharmonic function on $\C^m \times Y$, and suppose that $X := \left\{(t,z) \in \C^{m} \times Y : \rho(t,z) < 0\right\}$ is not necessarily a product manifold in $\C^{m} \times Y$. We assume further that for each $t$, the restriction $\rho^{[t]}$ of $\rho$ to 
$$X_t := \{z \in Y : (t,z) \in X\} \subset Y$$
takes values in $[-1,0)$. This assumption is satisfied when $X$ is a bounded strongly pseudoconvex domain, for instance.

Let $g$ be a Kähler metric for $Y$ and choose $\hat{g} := \pi^{*}_{\C^m}g_0 \oplus \pi^{*}_Y g$ as a Kähler metric for $\C^m \times Y$, where $g_0$ is the Euclidean metric on $\C^m$. Let $V \rightarrow X$ be a holomorphic vector bundle and let $h$ be Hermitian metric for $V \rightarrow X$ such that $\dbar_t\left(h^{-1}\del_Y h\right) = 0$, $\dbar_t\left(h^{-1}\del_t h\right) \geq 0$, and $h^{[t]} := \restr{h}{X_t}$ satisfies
\begin{equation}
    \Theta\left(h^{[t]}\right) + \left(\ricci(g) + \del_Y\dbar_Y\rho^{[t]}\right)\otimes\mathrm{Id}_{V^{[t]}} \geq_{\mathrm{Nak}} 0
\end{equation}
over $X_t$, for each $t \in \C^m$. Here, $\del_t$ and $\dbar_t$ denote the $\del$ and $\dbar$ operator with respect to the $t$-variable, on $\C^m$.\\

In \citep{berndtsson2017}, Berndtsson reduces the proof of his log-plurisubharmonic variation results from the case of pseudoconvex subdomains of product domains to product domains. Using a similar approach, we will show two additional log-plurisubharmonic variation results for the class of non-trivial families of Stein manifolds just described.

\subsubsection{Variations of Bergman kernels}
Let $K_t$ denote the Bergman kernel for the Bergman projection $L^2\left(X_t,h^{[t]},dV_g\right) \rightarrow \mathcal{H}^2\left(X_t,h^{[t]},dV_g\right)$ throughout the following proof of Theorem \ref{bergman-kernel-nontrivial-theorem}.

\begin{proof}
Let $z \in X_t$ and $\sigma \in \left(V^{[t]}_z\right)^{*}$ be fixed. Let us fix $t \in \C^m$ -- say $t = 0$. Assume for the moment that $m = 1$. Since the property is local, we may restrict ourselves to a neighborhood of $t = 0$. Let $U_0$ be a sufficiently small neighborhood of $0$ so that all the fibers $X_t$ are contained in a fixed pseudoconvex domain 
$Y_{\varepsilon} = \{\zeta \in Y : \rho(0,\zeta) < \varepsilon\}.$ As the sublevel set of a smooth plurisubharmonic function inside a Stein manifold, $X$ is a Stein manifold. Upon exhausting $X$ by an increasing sequence of relatively compact strongly pseudoconvex domains, we may assume that $X$ is a bounded strongly pseudoconvex domain with smooth boundary.\\

For $j \geq 4$, define 
$\rho_j := \dfrac{1}{j}\log\left(e^{j^2\rho}+1\right) \text{ and } h_j := he^{-\rho_j},$
so that $h_j \xrightarrow[j \rightarrow +\infty]{} h$ when $\rho \leq 0$ and $h_j \xrightarrow[j \rightarrow +\infty]{} 0$ when $\rho > 0$. Then $\Xi_{\delta,\eta}(h_j)$ equals
$$\begin{pmatrix}
\dbar_t\left(h^{-1}\del_t h\right) + \del_t\dbar_t\rho_j & \del_t\dbar_Y\rho_j \\
\del_Y\dbar_t\rho_j & \dfrac{\delta}{1+\delta}\left[\dbar_X\left(\left(h^{[t]}\right)^{-1}\del_X h^{[t]}\right) + \del_Y\dbar_Y\rho^{[t]}_j + \ricci(g) + \dfrac{4e^{\frac{1+\delta}{2}\eta}}{1+\delta}\del_Y\dbar_Y\left(-e^{-\frac{1+\delta}{2}\eta}\right)\right]
\end{pmatrix}.$$

Our goal is to find a function $\eta$ on $Y$, depending on $\rho^{[t]}$, so that $\Xi_{\delta,\eta}(h_j) \geq 0$ and
$$\Theta\left(h^{[t]}_j\right) + \left(\ricci(g) + \dfrac{4e^{\frac{1+\delta}{2}\eta}}{1+\delta}\del_Y\dbar_Y\left(-e^{-\frac{1+\delta}{2}\eta}\right)\right)\otimes \mathrm{Id}_V \geq_{\mathrm{Nak}} 0,$$
for each $t \in \C^m$.\\

Let $\mathfrak{a}, \mathfrak{b} \in \{t,Y\}$ as indices. Then:
$$\del_{\mathfrak{a}}\dbar_{\mathfrak{b}}\rho_j = \dfrac{je^{j^2\rho}}{1+e^{j^2\rho}}\del_{\mathfrak{a}}\dbar_{\mathfrak{b}}\rho + \dfrac{j^3 e^{j^2\rho}}{\left(1+e^{j^2\rho}\right)^2}\del_{\mathfrak{a}}\rho \wedge \dbar_{\mathfrak{b}}\rho.$$

Therefore, $\Xi_{\delta,\eta}(h_j)$ can be decomposed as
\begin{align*}
\begin{pmatrix}
\dbar_t\left(h^{-1}\del_t h\right) & 0 \\
0 & \dfrac{\delta}{1+\delta}\left(\dbar_X\left(\left(h^{[t]}\right)^{-1}\del_X h^{[t]}\right) + \ricci(g) + \del_Y\dbar_Y\rho^{[t]}\right)
\end{pmatrix} &+ \begin{pmatrix}
\del_t\dbar_t\rho_j & \del_t\dbar_Y\rho_j \\ \del_Y\dbar_t\rho_j & \del_Y\dbar_Y\rho_j
\end{pmatrix}\\
&+ \begin{pmatrix}
0 & 0 \\ 0 & M^{\eta,\rho}_j
\end{pmatrix}
\end{align*}
where
\begin{align*}
M^{\eta,\rho}_j &:= -\dfrac{1}{1+\delta}\left(\dfrac{je^{j^2\rho^{[t]}}}{1+e^{j^2\rho^{[t]}}}\del_Y\dbar_Y\rho^{[t]} + \dfrac{j^3 e^{j^2\rho^{[t]}}}{\left(1+e^{j^2\rho^{[t]}}\right)^2}\del_Y\rho^{[t]} \wedge \dbar_Y\rho^{[t]}\right)\\
&\ \ \ \, \, + \dfrac{\delta}{1+\delta}\left(2\del_Y\dbar_Y\eta-(1+\delta)\del_Y\eta\wedge\dbar_Y\eta-\del_Y\del_Y\rho^{[t]}\right).
\end{align*}

Our hypotheses clearly imply that
$$\Xi_{\delta,\eta}(h_j) \geq \begin{pmatrix}
0 & 0 \\ 0 & M^{\eta,\rho}_j
\end{pmatrix}$$ 
and
$$\Theta\left(h^{[t]}_j\right) + \left(\ricci(g) + \dfrac{4e^{\frac{1+\delta}{2}\eta}}{1+\delta}\del_Y\dbar_Y\left(-e^{-\frac{1+\delta}{2}\eta}\right)\right)\otimes \mathrm{Id}_V \geq_{\mathrm{Nak}} M^{\eta,\rho}_j\otimes\mathrm{Id}_V,$$
for each $t \in \C^m$.

Therefore, all we need to do is find a function $\eta$ on $Y$, depending on $\rho^{[t]}$, such that $M^{\eta,\rho}_j \geq 0$.\\

Now note that
$$\forall j > 0 : \dfrac{e^{j^2\rho^{[t]}}}{1+e^{j^2\rho^{[t]}}} < 1 \text{ and } \dfrac{e^{j^2\rho^{[t]}}}{\left(1+e^{j^2\rho^{[t]}}\right)^2} < \dfrac{1}{4}.$$

So, with $\eta := f\left(\rho^{[t]}\right)$ for some function $f$ to be determined shortly,
\begin{align*}
M^{\eta,\rho}_j &> \dfrac{\delta}{1+\delta}\left(-\dfrac{j}{\delta}\del_Y\dbar_Y\rho^{[t]} -\dfrac{j^3}{4\delta}\del_Y\rho^{[t]}\wedge\dbar_Y\rho^{[t]} + 2\del_Y\dbar_Y\eta-(1+\delta)\del_Y\eta\wedge\dbar_Y \eta-\del_Y\dbar_Y\rho^{[t]}\right)\\
&= \dfrac{\delta}{1+\delta}\dfrac{1}{\delta}\left(2\delta\del_Y\dbar_Y\eta-j\del_Y\dbar_Y\rho^{[t]}-\delta(1+\delta)\del_Y\eta\wedge\dbar_Y\eta-\dfrac{j^3}{4}\del_Y\rho^{[t]}\wedge\dbar_Y\rho^{[t]}-\delta \del_Y\dbar_Y\rho^{[t]}\right)\\
&=\dfrac{1}{1+\delta}\left(2\delta\del_Y\dbar_Y\eta-j\del_Y\dbar_Y\rho^{[t]}-\delta(1+\delta)\del_Y\eta\wedge\dbar_Y\eta-\dfrac{j^3}{4}\del_Y\rho^{[t]}\wedge\dbar_Y\rho^{[t]}-\delta \del_Y\dbar_Y\rho^{[t]}\right)\\
&= \dfrac{1}{1+\delta}\left[\left(2\delta f'\left(\rho^{[t]}\right)-j-\delta\right)\del_Y\dbar_Y\rho^{[t]}\right]\\
&\ \ \ \ \ \ +\dfrac{1}{1+\delta}\left[\left(2\delta f''\left(\rho^{[t]}\right)-\delta(1+\delta)\left(f'\left(\rho^{[t]}\right)\right)^2-\dfrac{j^3}{4}\right)\del_Y\rho^{[t]}\wedge\dbar_Y\rho^{[t]}\right].
\end{align*}

The function
$$f\left(\rho^{[t]}\right) := C_2 -\dfrac{2}{1+\delta}\log\left(\cos\left(\dfrac{\sqrt{1+\delta}j^{3/2}}{4\sqrt{\delta}}(\rho^{[t]}+C_1\delta)\right)\right)$$
satisfies
$$2\delta f''\left(\rho^{[t]}\right)-\delta(1+\delta)\left(f'\left(\rho^{[t]}\right)\right)^2-\dfrac{j^3}{4} = 0,$$
for any constants $C_2$ and $C_1$. The constant $C_1$ will be determined later and we can just let $C_2 = 0$. Now,
$$f'\left(\rho^{[t]}\right) = \dfrac{j^{3/2}}{2\sqrt{\delta(1+\delta)}}\tan\left(\dfrac{j^{3/2}}{4}\sqrt{\dfrac{1+\delta}{\delta}}(\rho^{[t]}+C_1\delta)\right),$$
and so
\begin{align*}
&-\dfrac{j}{\delta}\del_Y\dbar_Y\rho^{[t]} -\dfrac{j^3}{4\delta}\del_Y\rho^{[t]}\wedge\dbar_Y\rho^{[t]} + 2\del_Y\dbar_Y\eta-(1+\delta)\del_Y\eta\wedge\dbar_Y \eta-\del_Y\dbar_Y\rho^{[t]}\\
&= \dfrac{1}{\delta}\left(\sqrt{\dfrac{\delta j^3}{1+\delta}}\tan\left(\dfrac{j^{3/2}}{4}\sqrt{\dfrac{1+\delta}{\delta}}(\rho^{[t]}+C_1\delta)\right)-(j+\delta)\right)\del_Y\dbar_Y\rho^{[t]}.
\end{align*}

For each $j \geq 4$, we can always find some $\delta := \delta_j > 0$ such that
$$\sqrt{\dfrac{\delta j^3}{1+\delta}} > j+\delta.$$

Furthermore, if we let $C_1 := C_{\delta,j} := \dfrac{1}{\delta}\left(\pi\sqrt{\dfrac{\delta}{j^3(1+\delta)}}+1\right)$, then $C_1\delta \leq 1+\pi/8$ for $j \geq 4$, whence $\dfrac{j^{3/2}}{4}\sqrt{\dfrac{1+\delta}{\delta}}(\rho^{[t]}+C_1\delta)$ takes values in $[\pi/4,\pi/2)$ since $-1 \leq \rho^{[t]} < 0$. Hence,
$$\sqrt{\dfrac{\delta j^3}{1+\delta}}\tan\left(\dfrac{j^{3/2}}{4}\sqrt{\dfrac{1+\delta}{\delta}}(\rho^{[t]}+C_1\delta)\right)-(j+\delta) > 0.$$

Therefore, Theorem \ref{bergman-kernel-trivial-theorem} applies to our situation where the product domain is $U_0 \times Y_\varepsilon$. Hence, by Proposition \ref{prop-bergman-approx} (combined with Proposition \ref{runge-l2-approx-prop}) $\log\left\langle\sigma\otimes\bar{\sigma},K_t(z,\bar{z})\right\rangle$, over a relatively compact strongly pseudoconvex subdomain of $X$, can be written as an increasing limit of functions that are subharmonic with respect to $t$. We then conclude, by Proposition \ref{upper-semilemma-bergman}, that the function $t \mapsto \log\left\langle\sigma\otimes\bar{\sigma},K_t(z,\bar{z})\right\rangle$ is also subharmonic by upper semicontinuity. Again, by upper semicontinuity, we get that $t \mapsto \log\left\langle\sigma\otimes\bar{\sigma},K_t(z,\bar{z})\right\rangle$ is plurisubharmonic if $m \geq 1$ since its restriction to any line is subharmonic. Finally, Proposition \ref{ramadanov-manifold-bergman-union-domains} implies that the function $t \mapsto \log\left\langle\sigma\otimes\bar{\sigma},K_t(z,\bar{z})\right\rangle$, over $X$, is the decreasing limit of a sequence of plurisubharmonic functions and is thus plurisubharmonic.\\

Of course, the plurisubharmonicity is strict if the twisted curvature conditions are strict. However, we will not repeatedly state this in what follows.\\

So far, we have shown that $t \mapsto \log\left\langle\sigma\otimes\bar{\sigma},K_t(z,\bar{z})\right\rangle$ is plurisubharmonic for $z$ fixed. We will now show that for every $\sigma \in \left(V^{[t]}_z\right)^{*}$, $(t,z) \mapsto \log\left\langle\sigma\otimes\bar{\sigma},K_t(z,\bar{z})\right\rangle$ is plurisubharmonic. As before, we can choose a sufficiently small neighborhood $U_0$ of $0$ such that all the fibers $X_t$ are contained in a fixed pseudoconvex domain $Y_{\varepsilon}$.

In general, we can find a holomorphic tangent vector field $\mathfrak{F}$ on $X$ such that $d\pi_{\C^m}(\mathfrak{F}) = \del/\del t$ (see \citep[Lemma 2.3]{berndtsson-paun-2008}).
Let $\Phi_{\mathfrak{F}}$ denote the flow of $\mathfrak{F}$ and let $\Phi^t_{\mathfrak{F}}$ denote the flow at time $t$. Let $\tilde{X}_t := \Phi^t_{\mathfrak{F}}(X_t)$. Since $\Phi^t_{\mathfrak{F}}$ maps $X_t$ to $\tilde{X}_t$ biholomorphically, the invariance of Bergman kernels implies that
$K_t\left(z,\bar{z}\right) = \tilde{K}_t\left(\Phi^{t}_{\mathfrak{F}}(z),\overbar{\Phi^{t}_{\mathfrak{F}}(z)}\right)$ for any $z \in X_t$. Here $K_t$ is the Bergman kernel for the projection $L^2\left(X_t,h^{[t]},dV_g\right) \rightarrow \mathcal{H}^2\left(X_t,h^{[t]},dV_g\right)$ and $\tilde{K}_t$ is the one for the projection $L^2\left(\tilde{X}_t,\left(\Phi^{-t}_{\mathfrak{F}}\right)^{*}h^{[t]},\left(\Phi^t_{\mathfrak{F}}\right)_{*}dV_g\right) \rightarrow \mathcal{H}^2\left(\tilde{X}_t,\left(\Phi^{-t}_{\mathfrak{F}}\right)^{*}h^{[t]},\left(\Phi^t_{\mathfrak{F}}\right)_{*}dV_g\right).$

We are now in the previous situation, and so for any fixed $z \in X_t$, the function
$$t \mapsto \log\left\langle \sigma \otimes \bar{\sigma}, K_t\left(z,\bar{z}\right)\right\rangle = \log\left\langle \sigma \otimes \bar{\sigma}, \tilde{K}_t\left(\Phi^{t}_{\mathfrak{F}}(z),\overbar{\Phi^{t}_{\mathfrak{F}}(z)}\right)\right\rangle$$ is plurisubharmonic for any $\sigma \in \left(V^{[t]}_{z}\right)^{*}$. We may again assume that $m = 1$ without loss of generality by the previous upper semicontinuity arguments.

The flow $\Phi^{t}_{\mathfrak{F}}$ evaluated at $z$ has the first-order Taylor expansion $\Phi^{t}_{\mathfrak{F}}(z) = z+t\mathfrak{F}(z)+O(t^2)$ in $t$ around $t = 0$. Therefore, the function 
$t \mapsto \log\left\langle \sigma \otimes \bar{\sigma}, \tilde{K}_t\left(z,\bar{z}\right)\right\rangle$ is subharmonic in the direction of $\mathfrak{F}$, which is an arbitrary lift of $\del/\del t$. We thus have the the subharmonicity of the function $t \mapsto \log\left\langle \sigma \otimes \bar{\sigma}, \tilde{K}_t\left(z,\bar{z}\right)\right\rangle$ in a general non-vertical direction in $X$. 

Now let $z = \Phi^{-t}_{\mathfrak{F}}(w)$. Then the function $t \mapsto \log\left\langle \sigma \otimes \bar{\sigma}, K_t\left(\Phi^{-t}_{\mathfrak{F}}(w),\overbar{\Phi^{-t}_{\mathfrak{F}}(w)}\right)\right\rangle$
is subharmonic. Since the vector field $\mathfrak{F}$ is an arbitrary lift of $\del/\del t$, this shows that the function $(t,z) \mapsto \log\left\langle\sigma\otimes\bar{\sigma},K_t(z,\bar{z})\right\rangle$
is subharmonic in every non-vertical direction. In the vertical directions $z \mapsto \log\left\langle\sigma\otimes\bar{\sigma},K_t(z,\bar{z})\right\rangle$ is trivially subharmonic, as it is the sum of squares of holomorphic functions. This completes the proof.
\end{proof}

\subsubsection{Variations of families of compactly supported measures}

Now let $\left\{\hat{\mu}_t\right\}_{t\in U}$ be a family of $\left(V^{[t]}\right)^{*}$-valued complex measures over $X_t$ that are all locally supported in a compact subset of $X$. For each section $f \in \Gamma(E_h)$, define the measure $\mu^{(f)}_t = \left\langle f^{[t]},\hat{\mu}_t\right\rangle$ and define the mapping $\xi^{(\mu)}_t$ by
$$f^{[t]} \mapsto \left\langle\xi^{(\mu)}_t,f^{[t]}\right\rangle := \mu^{(f)}_t(X_t) = \int_{X_t} \left\langle f^{[t]},\hat{\mu}_t\right\rangle.$$

Then, similarly to the case of trivial families of Stein manifolds, $\xi^{(\mu)}$ defines a smooth section of $E^{*}_h$. We now prove Theorem \ref{family-measures-nontrivial-theorem}.

\begin{proof}
By the standard exhaustion argument, we may assume that $X$ is a bounded strictly pseudoconvex domain with smooth boundary. Since the result is local, we may after restricting $t$ to lie in a small neighborhood $V_0$ of a given point, say $0$, assume that $X \subset V_0 \times Y_{\varepsilon}$ where $Y_{\varepsilon} := \{z \in Y : \rho(0,z) < \varepsilon\}$ is a pseudoconvex domain in $Y$. Then we apply Theorem \ref{family-measures-trivial-theorem} to $V_0 \times Y_{\varepsilon}$ with $h$ replaced by $h_j$, where $h_j$ is the approximating metric used in the proof of Theorem \ref{bergman-kernel-nontrivial-theorem}. Recall that $h_j := h e^{-\rho_j}$ where $\rho_j$ is a function of $\rho$ such that $e^{-\rho_j}$ converges to $1$ as $j \rightarrow +\infty$ when $\rho \leq 0$, and $e^{-\rho_j}$ converges to $0$ as $j \rightarrow +\infty$ when $\rho > 0$. Therefore,
$$\norm{f}^2_{L^2\left(\{t\} \times Y_{\varepsilon},h^{[t]}_j\right)} := \int_{\{t\} \times Y_{\varepsilon}} \abs{f}^2_{h^{[t]}_j} dV_g \xrightarrow[j \rightarrow +\infty]{} \int_{X_t} \abs{f}^2_{h^{[t]}} dV_g =: \norm{f}^2_{L^2\left(X_t,h^{[t]}\right)},$$
and similarly,
$$\abs{\abs{\int_{X_t}\left\langle f^{[t]},\hat{\mu}_t\right\rangle} - \abs{\int_{\{t\}\times Y_{\varepsilon}}\left\langle f^{[t]},\hat{\mu}_t\right\rangle}} \leq \abs{\int_{X_t}\left\langle f^{[t]},\hat{\mu}_t\right\rangle - \int_{\{t\}\times Y_{\varepsilon}} \left\langle f^{[t]},\hat{\mu}_t\right\rangle} \xrightarrow[j \rightarrow +\infty]{} 0,$$
so that
$$\norm{\xi^{(\mu)}}^2_{*,h^{[t]}_j,\{t\}\times Y_{\varepsilon}} \xrightarrow[j \rightarrow +\infty]{} \norm{\xi^{(\mu)}}^2_{*,h^{[t]},X_t}$$
and so the theorem follows.
\end{proof}

Once again, letting $\hat{\mu}_t := \sigma \otimes \delta_{F(t)}$ for each $t \in U$, where $F$ is a holomorphic map from $U$ to $X_t$, $\delta_{F(t)}$ denotes a point-mass measure supported at $F(t)$ and $\sigma \in V^{*}_{F(t)}$, we see that under the hypotheses of Theorem \ref{bergman-kernel-nontrivial-theorem}, the function $U \ni t \mapsto \log\left\langle \sigma \otimes \bar{\sigma},K_t\left(F(t),F(t)\right)\right\rangle$ is plurisubharmonic (or strictly so) or identically $-\infty$ for every $\sigma \in \left(V^{[t]}\right)^{*}_{F(t)}$.

\section*{Acknowledgements}
The author is very grateful to his doctoral advisor Prof. Dror Varolin for suggesting the problem of the variation of Bergman spaces for unbounded Stein manifolds, which led to the work presented in this article. The author is also grateful to Prof. Dror Varolin for his continued guidance until the completion of this project.

\printbibliography
\lipsum[184-187]
\end{document}